\documentclass[11pt]{article}
\usepackage[utf8]{inputenc}
\usepackage[pagebackref=true]{hyperref}
\usepackage{amssymb,amsmath,bm,bbm,amsthm,microtype,geometry,xcolor,tocbibind,enumitem}
\usepackage{caption,subcaption}
\usepackage{graphicx,color,tikz}

\definecolor{DarkBlue}{rgb}{0.15,0.15,0.55}
\hypersetup{linktocpage=true,colorlinks=true,allcolors=DarkBlue}
\usepackage[capitalise]{cleveref}

\usepackage[OT2,T1]{fontenc}
\DeclareSymbolFont{cyrletters}{OT2}{wncyr}{m}{n}
\DeclareMathSymbol{\Sha}{\mathalpha}{cyrletters}{"58}

\newcommand{\R}{\mathbb{R}}
\newcommand{\Z}{\mathbb{Z}}
\newcommand{\T}{\mathbb{T}}
\newcommand{\N}{\mathbb{N}}

\newcommand{\Lop}{\mathrm{L}}

\newcommand{\SpT}{\mathcal{S}'(\mathbb{T})}
\newcommand{\ST}{\mathcal{S}(\mathbb{T})}

\usepackage{accents}

\newtheorem{proposition}{Proposition}[section]
\newtheorem{theorem}[proposition]{Theorem}

\newtheorem{corollary}[proposition]{Corollary}
 
\newtheorem{definition}[proposition]{Definition}
\usepackage{titlesec}
\titleformat*{\paragraph}{\itshape\mdseries}

\title{The Critical  Smoothness  of Generalized Functions}
\author{Julien Fageot and John Paul Ward}
\begin{document}

\maketitle

\begin{abstract}
    For each integrability parameter $p \in (0,\infty]$, the critical  smoothness of a periodic  generalized function $f$, denoted by $s_f(p)$ is the supremum over the smoothness parameters $s$ for which $f$ belongs to the Besov space $B_{p,p}^s(\mathbb{T})$ (or other similar function spaces).
    This paper investigates the evolution of the critical smoothness with respect to the integrability parameter $p$. Our main result is a simple characterization of all the possible \emph{critical smoothness functions} $p \mapsto s_f(p)$ when $f$ describes the space of generalized periodic functions.
    We moreover characterize the compressibility of generalized periodic functions in wavelet bases from the knowledge of their critical smoothness function.
\end{abstract}

\section{Generalized Functions and their Critical Smoothness} 

    In functional analysis, continuous-domain (generalized) functions are classified with respect to their regularity properties. 
    The latter can be measured in different smoothness classes such as Sobolev, Hölder, or Besov spaces, to name a few. The regularity of a class of functions deeply influences how well a given function can be approximated in adequate bases (\emph{e.g.}, Fourier or wavelets). 
    Our goal in this note is to use Besov spaces to characterize the smoothness properties of functions. 
    We will introduce the \emph{critical smoothness function} of a generalized function (see Definition \ref{def:critical}) and highlight its main properties. 
    
    \textit{Preliminary remark.} This paper will be centered on \emph{periodic} functions, in order to focus on the smoothness, which is a local property. 
    The periodic framework is only considered for its convenience (it excludes questions regarding the asymptotic decay/growth properties of functions). 
    
    \subsection{The Critical Smoothness}
    \label{sec:defspf}
        
    Let $\T = \R / \Z = [0,1]$ be the $1$-dimensional torus, where the extremeties $0$ and $1$ are identified.
    The space of periodic generalized function $\SpT$ is the topological dual of the space of periodic and infinitely smooth functions $\ST$ endowed with its usual Fréchet topology~\cite{Treves1967}. 
    Besov spaces will be formally defined in Section \ref{sec:prelim}. For the moment, it suffices to recall that they are subspaces $B_{p,q}^s(\T)$ of $\SpT$ such that, roughly speaking,
    $f \in B_{p,q}^s(\T)$  means that $f$ has $s$ derivative in $L^p(\T)$, 
    at least for $s \in \N$.
    We call $p \in (0,\infty]$ the integrability parameter and $s\in \R$ the smoothness parameter. The last parameter $q \in (0,\infty]$ plays a secondary role (we shall mostly consider the cases $q=p$ hereafter). 
    Besov spaces are Banach spaces (quasi-Banach spaces, respectively) for any $p,q \geq 1$ (for $0<p<1$ or $0<q<1$, respectively)~\cite[Theorem 1, Section 3.5.1]{schmeisser87}.  \\
    
    For any fixed $0< p , q \leq \infty$, we have the projective and inductive limits
    \begin{equation} \label{eq:SSpviaB}
    \ST = \bigcap_{s\in\R}  B_{p,q}^s(\T) \quad \text{ and } \quad
   \SpT = \bigcup_{s\in\R}  B_{p,q}^s(\T),
     \end{equation}
    as proved for instance in~\cite{Kabanava2008tempered}\footnote{Kabanava considers the case of tempered generalized functions but the result easily applies to the periodic setting.}. Moreover, we have the topological embedding (see Proposition \ref{prop:embeddings})
    \begin{equation} \label{eq:embedsmooth}
        B_{p,q}^{s_2}(\T) \subset B_{p,q}^{s_1}(\T)
    \end{equation}
    for any $s_1 \leq s_2$. These two facts lead us to the following definition.
    
    \begin{definition} \label{def:critical}
    Let $f \in \SpT$ and $0 < p,q \leq \infty$. Then, the \emph{$(p,q)$-critical smoothness} of $f$ is defined by
    \begin{equation} \label{eq:cs}
        s_f(p,q) = \sup \{ s \in \R, \ f \in B_{p,q}^s(\T) \} \in (-\infty, \infty].
    \end{equation}
    \end{definition}
     
     Note that $s_f(p,q)$ is well-defined, since the right relation \eqref{eq:SSpviaB} implies that the supremum in \eqref{eq:cs} is taken over a non-empty set. Moreover, the embeddings \eqref{eq:embedsmooth} implies that $f \in B_{p,q}^{s}(\T)$ for any $s < s_f(p,q)$ and that $f \notin B_{p,q}^{s}(\T)$ for any $s > s_f(p,q)$ (a generalized function $f$ can belong to the critical space $B_{p,q}^{s_f(p,q)}(\T)$ or not). 
    The parameter $q$ only plays a secondary role for embedding properties (see Proposition \ref{prop:embeddings} for a precise meaning). A first simple consequence is the following result. 
    
    \begin{proposition} \label{prop:removeq}
     Let $f \in \SpT$ and $0 < p \leq \infty$. Then, for any $0 < q_1 , q_2 \leq \infty$, we have that
     \begin{equation}
         s_f(p,q_1) = s_f(p,q_2).
     \end{equation}
    \end{proposition}
    
    \begin{proof}
    {\color{black}
    Let $f\in \SpT$ with $p,q_1,q_2$ as stated above, and let $s_n$ be an increasing sequence converging to $s_f(p,q_1)$. Then since 
    $f\in B_{p,q_1}^{s_n}(\T)$ for all $n$, 
    we have by the embedding properties 
    $f\in B_{p,q_2}^{s_n-1/n}(\T)$. Therefore 
    $s_f(p,q_2)\geq \sup{s_n-1/n} = s_f(p,q_1)$. 
    Similarly, $s_f(p,q_1)\geq  s_f(p,q_2)$. 
    }
    \end{proof}
    
    \begin{definition}
    Let $f \in \SpT$. We denote by $s_f(p) \in (-\infty, \infty]$ the common value of the $s_f(p,q)$, $0 < q \leq \infty$. 
     The function $p \mapsto s_f(p)$ is called the \emph{critical smoothness function} of $f$. 
    \end{definition}

    It is only possible to achieve an infinite local smoothness for infinitely smooth functions, as stated in the following result.
    
    \begin{proposition} \label{prop:removeS}
    Let $f \in \SpT$.
    If $f \in \ST$, then $s_f(p) = \infty$ for every $0< p \leq \infty$. Conversely, if there exists some $0 < p_0 \leq \infty$ such that $s_f(p_0) = \infty$, then $s_f(p) = \infty$ for every $0 < p \leq \infty$ and we have $f \in \ST$. 
    \end{proposition}
    
    \begin{proof}
    The first implication is direct consequence of the left relation \eqref{eq:SSpviaB}. For the converse, we observe that for each $0 <  p \leq \infty$ and any $s \in \R$, there exists $s_0 \in \R$ such that $B_{p_0}^{s_0}(\T) \subset B_p^s(\T)$ (according to Proposition \ref{prop:embeddings}, one can select $s_0 > s +\left( \frac{1}{p_0} - \frac{1}{p} \right)_+ $ with $x_+ = \max(x,0)$). Hence, since $s_f(p_0) = \infty$, we have that $f \in B_{p_0}^{s_0}(\T)$ and therefore $f \in B_p^s(\T)$. This is true for any $s\in \R$, implying that $s_f(p) = \infty$. With the left relation in \eqref{eq:SSpviaB}, this implies that $f \in \ST$.  
    \end{proof}
    
    \textit{Remark.} Thanks to Proposition \ref{prop:removeS}, one can discard infinitely smooth functions, and therefore restrict to generalized functions for which  $s_f(p) \in \R$ is necessarily finite for every $0 < p \leq \infty$. \\
    
        We expect the critical smoothness function to be a useful mathematical concept to help in describing and characterizing the smoothness properties of functions and random processes. This motivates our study. The main question addressed in this paper is the following: What is the possible evolution of the critical smoothness $p \mapsto s_f(p)$ when $f$ is in $\SpT$? In other terms, we aim at identifying the class of critical smoothness functions. 

    \subsection{Contributions and Outline}
    
    In a much better way than Fourier bases, wavelet bases have shown to efficiently characterize function spaces. Their key aspect is that they are unconditional bases for most of the classical function spaces, including Besov spaces~\cite{MeyerWaO}. One can therefore characterize the fact that $f$ lies in a given Besov spaces via simple conditions (\emph{e.g.}, the finiteness of a weighted $\ell_p$-norm) on its wavelet coefficients~\cite{Triebel2008function}, hence determining the critical smoothness functions.
    In this paper, we follow this line of research:  we use Besov spaces to characterize the smoothness properties of functions and use wavelet methods to characterize this Besov regularity.
    Our contributions are the following.
    
    \begin{itemize}[label=\raisebox{0.30ex}{\tiny$\bullet$}]
    
          \item[(i)] Our main result is Theorem \ref{theo:main}, where we fully characterize the functions $s$ that are the critical smoothness functions of some $f\in \SpT$. We show that the class of critical smoothness functions coincides with the class of functions $s$ such that $\frac{1}{p} \mapsto s(p)$ is increasing, concave, and $1$-Lipschitz over $[0, \infty)$. 
        
        \item[(ii)] In addition to our main theorem, we prove various properties of the critical smoothness functions: right and left differentiability, slopes evolution, behavior of $s_{f+g}(p)$ in terms of $s_f(p)$ and $s_g(f)$, etc.
        
        \item[(iii)] We show how to connect the critical smoothness function to the wavelet compressibility of a given generalized function $f$, which measures the speed of decay of the best $N$-term wavelet approximation of $f$ in a given Besov space. We demonstrate how the critical smoothness function is sufficient to determine this compressibility in Theorem \ref{theo:compress}. 
        
    \end{itemize}
    
    The paper is organized as follows. 
    We position our contributions in relation to other works for deterministic and random functions in Section \ref{sec:related}. 
    We formally define periodic Besov spaces using wavelet methods in Section \ref{sec:prelim}. 
    Our main result is presented and proved in Section \ref{sec:mainresult}, together with interesting properties of critical smoothness functions. 
    In Section \ref{sec:compressibility}, we discuss the compressibility of a generalized function $f$ in wavelet bases and connect it to its critical smoothness function $s_f(p)$.
    Finally, we discuss our results and conclude in Section \ref{sec:conclusion}.
    
    \section{Related Works and Examples}
    \label{sec:related} 


        The traditional theory of function spaces classifies functions with respect to their smoothness properties, measured in terms of some integrability parameter $0<p\leq \infty$.
        One of the major achievement of the global theory of function spaces in the 20th century has been to define general function classes capturing most of the smoothness, approximation, or integrability properties of functions, notably with the introduction of Sobolev $W^s_p$ and Besov spaces $B_{p,q}^s$~\cite{schmeisser87}. 
        Many works in functional analysis and stochastic processes have dealt with the determination of the Besov regularity of some (random) function. Some of these results can be re-interpreted in terms of the critical smoothness function of the studied function. We provide some examples.


        \subsubsection*{The Critical Smoothness of Deterministic Functions}
        
                
        Perhaps the simplest example of a generalized function for which one can characterize the Besov regularity is the Dirac impulse $\delta$. Its critical function is given by $s_\delta(p) = \frac{1}{p} - 1$~\cite[p. 164]{schmeisser87};  see also \cite[Proposition 5]{aziznejad2020wavelet} for a wavelet-based proof. From this, one deduce that the critical smoothness of piecewise constant functions $f$ is $s_f(p)= \frac{1}{p}$.\\
        
        \textit{Local Hölder regularity:} Several works deal with the \emph{local} Hölder regularity of functions: for a continuous function $f : \R \rightarrow \R$, one can ask what is the local Hölder regularity $h_f(x_0)$ of $f$ at $x_0 \in \R$. 
        The function $x_0 \mapsto h_f(x_0)$ is called the \emph{local Hölder function}~\cite{seuret2002local} of $f$ and the possible evolution of $h_f$ for continuous $f$ has been characterized~\cite{andersson1997characterization,daoudi1998construction,jaffard1995functions}, with extensions to non-continuous functions~\cite{ayache2010holder}.
        The use of wavelet methods in this context is well established~\cite{jaffard1996wavelet}. 
        In comparison, we consider \emph{uniform} smoothness properties: we look for the critical Hölder regularity $s_f(\infty)$ of $f$ such that $f$ is uniformly $s$-Hölder (\emph{i.e.}, $s$-Hölder at any $x_0$) for any $s < s_f(\infty)$. The local Hölder function is linked to the critical smoothness function via the relation $s_f(\infty) = \inf_{x_0} h_f(x_0)$: The uniform Holder regularity of $f$ is the worst case of the local smoothness of $f$. \\
        
        \textit{Fractals and PDEs:} The study of the Besov regularity is interesting for functions with limited smoothness that naturally arise in fractals~\cite{Mandelbrot1982fractal} and PDEs.
        Some authors have studied the Besov regularity of fractal functions.
        In~\cite{jaffard1996local}, Stéphane Jaffard and Benoît Mandelbrot considered space-filling fractal functions and determined their Hölder regularity using wavelet techniques.    
        General classes of fractal functions are introduced in~\cite{massopust2005splines,massopust2016local} and shown to be in some Besov spaces, which gives some lower bounds for $s_f(p)$.
        The connection between fractals, wavelets and smoothness function spaces is developed in~\cite[Chapter 12]{massopust2016fractal}. 
        Besov spaces are also used to characterize the regularity of the solutions of partial differential equations~\cite{dahlke1997besov,hansen2014n}, including non-linear ones~\cite{dahlke2020besov}. In this context, one main motivation is the link between the best $N$-term wavelet approximation and the Besov regularity of functions~\cite{Devore1998Nterm} for the approximation of solutions of PDEs. 
        
        \subsubsection*{The Critical Smoothness of Random Functions} 

        \textit{The Brownian motion:} Understanding the sample-path properties of random models has attracted considerable attention since the pioneering works of Paul Lévy~\cite{levy1965processus}.
        Historically, investigations started with the regularity of the sample paths of the Brownian motion $B$, which is Hölder continuous of order $\alpha$ if and only if $\alpha < 1/2$. With our notation, this corresponds to the parameter $p = \infty$  and we have\footnote{The historical works on the Hölder regularity of the Brownian motion were neither formulated for Besov spaces $B_{\infty,\infty}^s$ nor in the periodic setting. Nevertheless, it is possible to convert those results in our setting,  as discussed in Section~
        \ref{sec:beyondbesov}.} 
        $s_B(\infty) = \frac{1}{2}$.
        The Besov regularity of the Brownian motion and its extension has been studied in~\cite{ciesielski1993orlicz,Ciesielski1993quelques,hytonen2008,Roynette1993}. 
        The Gaussian white noise $W$, which is simply the (weak) derivative of the Brownian motion, is strongly related since we have $s_B(p) = s_W(p) + 1$ for any $p>0$. It has been studied by Veraar over the torus for $p \geq 1$ with Fourier domain techniques in~\cite{Veraar2010regularity} and completed for $0 < p < 1$ in~\cite{aziznejad2020wavelet}. These works allow one to deduce that the Brownian motion satisfies
        \begin{equation}\label{eq:sBp}
        s_B(p) = \frac{1}{2}, \quad \forall 0< p \leq \infty. 
        \end{equation}
         More generally, the fractional Brownian motion $B_H$ with Hurst index $H \in (0,1)$~\cite{Mandelbrot1968} is such that $s_{B_H}(p) = H$ for every $0 < p \leq \infty$~\cite[Theorem IV.3]{Ciesielski1993quelques}. This fact has been generalized for any Gaussian process $X$ such that $\Lop_\gamma X = W$ with $\Lop$ a $\gamma$-admissible operators (such as the $\gamma$th order derivative, see \cite[Definition 8]{Fageot2017nterm}) and $W$ a Gaussian white noise, whose Besov regularity is~\cite[Corollary 1]{Fageot2017nterm} $s_X(p) = \gamma - \frac{1}{2}, \quad \forall 0 < p \leq \infty$.     This highlights a specificity of classical Gaussian models, for which the critical Besov function is constant with respect to $p$. \\

        \textit{Lévy processes and their extensions:} 
        More generally, several authors have considered the class of Lévy processes $L$, which generalizes the Brownian motion by relaxing the Gaussian hypothesis~\cite{Sato1994levy}. Their Besov regularity has been considered by René Schilling~\cite{Schilling1997Feller,Schilling1998growth,Schilling2000function} and Volken Herren~\cite{Herren1997levy}. They gave sufficient conditions ensuring that $L$ belongs to a given Besov space. Sharp results regarding the Besov smoothness are obtained in~\cite{aziznejad2018wavelet,Fageot2017besov,Fageot2017multidimensional}.
        One can summarize the results by saying that, under mild conditions\footnote{Technically, we require that two Blumenthal--Getoor indices are equal, see~\cite{aziznejad2018wavelet}, especially Eq. (12), for more details.}, for a non-Gaussian Lévy process, we have
        \begin{equation}\label{eq:sLp}
        s_L(p) = \frac{1}{\max(p, \beta)},
        \end{equation}
        where $\beta \in [0,2]$ is the Blumenthal-Getoor index of $L$, which is known to characterize many of the local properties of a Lévy process~\cite{Blumenthal1961sample,fageot2017gaussian}. This includes compound Poisson processes $P$, for which $\beta = 0$ and therefore $s_{P} = \frac{1}{p}$ and non-Gaussian $\alpha$-stable processes~\cite{Taqqu1994stable} $L_\alpha$ for which $\beta = \alpha \in (0,2]$ and hence $s_{L_\alpha} (p ) = \frac{1}{\max(p,\alpha)}$. 
        This has been generalized for random processes $X$ that are solutions of stochastic differential equations with Lévy white noise $W$ as $\Lop_\gamma X = W$, where $\Lop_\gamma$ is a $\gamma$-admissible operator, showing that $ s_{X}(p) = (\gamma - 1 )  +  \frac{1}{\max(p, \beta)}$~\cite[Corollary 1]{Fageot2017nterm}. This is consistent with \eqref{eq:sLp} for which $\gamma = 1$.\\

        \textit{Other random models:} More generally, Besov spaces are used to characterize the regularity of the solutions of stochastic partial differential equations driven by Gaussian or L\'evy white noises~\cite{chong2019path,cioica2012spatial,Cioica2015besov,hummel2019stochastic,hummel2020sample,lindner2011approximation}.
        Many results in the literature can be reformulated in terms of the critical smoothness function of the random processes, as we did for the Brownian motion and for Lévy processes.  
        Random models with lacunary wavelet series have also been considered~\cite{aubry2002random,jaffard2000lacunary}. For instance, \cite[Theorem 1]{bochkina2013besov} considers random processes $f$ constructed via lacunary random wavelet sequences and such that 
        $s_f(p) = s_0 + \frac{\alpha_0}{p}$
        with $0\leq \alpha_0 <1$ and $s_0 > 0$. Such critical smoothness evolution will play a crucial role to prove the main result of this paper. \\
        
        \textit{Beyond uniform smoothness:} All these contributions deal with the \emph{uniform} smoothness of random processes. This 
        is especially relevant for random models with stationary properties, for which the local smoothness is constant and therefore equal to the uniform smoothness. Even if this is beyond the scope of this paper, we mention the existence of important generalizations for which the local smoothness of the random process evolves, as is the case for multifractal Gaussian~\cite{ayache1999generalized} and Lévy processes~\cite{barral2007singularity} and for Lévy-type processes~\cite{Bottcher2014levy}, to name a few. 
    
        \subsubsection*{Besov Regularity and Wavelet Approximation}
    
        One of the main achievements of the theory of Besov spaces has been to recognize that they characterize the speed of convergence of their best $N$-term wavelet approximation~\cite{Devore1998Nterm,Cohen2003numerical}. This correspondence is made possible due to the fact that the Besov regularity is captured by sequence norms on wavelet series~\cite{MeyerWaO,Garrigos2004sharp}. It therefore comes as no surprise that the critical smoothness function is sufficient to characterize the wavelet compressibility, as detailed in Section~\ref{sec:compressibility}. The use of the Besov regularity of Lévy white noises and Lévy processes has been used to deduce the rate of the $n$-term approximation of Lévy processes in~\cite{Fageot2017nterm}, with a special emphasis on compound Poisson processes with elementary tools in~\cite{aziznejad2020wavelet}.
        
    \section{Mathematical Preliminaries} 
    \label{sec:prelim}

       \subsection{Periodic Besov Spaces}

        Besov spaces allow one to measure the smoothness properties of functions in different $L_p$ scales, including $p= \infty$ (Hölder-type regularity) or $p=2$ ($L_2$-Sobolev regularity). 
        In the periodic setting, they can be defined based on Fourier series as follows. We refer to~\cite[Section 3.5]{schmeisser87} for more details. 
        
        The Fourier series of a periodic generalized function $f \in \SpT$, is written as $(\widehat{f}_k)_{k\in \Z}$.
        We denote by $\mathrm{Supp}\{f\}$ the support of a function $f$.
        We fix two non-negative functions $\varphi_0$ and $\varphi_1$ in the Schwartz class $\mathcal{S}(\R)$ of infinitely smooth and rapidly decaying functions such that
        \begin{equation}
            \mathrm{Supp} \{\varphi_0\} \subset \{ x \in \R, \ |x|\leq 2\}, \quad \mathrm{Supp} \{\varphi_1\} \subset \{ x \in \R, \ \frac{1}{2}\leq |x|\leq 2 \},  \quad \text{and} \quad  \sum_{j \geq 0} \varphi_j = 1
        \end{equation}
        with $\varphi_j = \varphi_1(2^{-j} \cdot)$
         for all $j \geq 2$. More information on such systems $\bm{\varphi} = (\varphi_j)_{j\geq0}$, that are known to exist, can be found in~\cite[Section 2.1.1]{schmeisser87}.
         
         \begin{definition}
          Let $0<p,q\leq \infty$ and $s \in \R$. The Besov space $B_{p,q}^s(\T)$ is the space of periodic generalized functions $f$ such that 
          \begin{equation} \label{eq:normqfinite}
              \lVert f \rVert_{B_{p,q}^s} = \left( \sum_{j \geq 0} 2^{jsq} \left\lVert \sum_{k\in \Z} \varphi_j(k) \widehat{f}_k \mathrm{e}^{2 \mathrm{i} \pi k \cdot} \right\rVert_{L_p}^{q} \right)^{1/q} < \infty,
          \end{equation}
          when $q < \infty$, and 
          \begin{equation} \label{eq:normqinfty}
              \lVert f \rVert_{B_{p,\infty}^s} =  \sup_{j\geq 0} 2^{js} \left\lVert \sum_{k\in \Z} \varphi_j(k) \widehat{f}_k \mathrm{e}^{2 \mathrm{i} \pi k \cdot} \right\rVert_{L_p} < \infty.
          \end{equation}         
         \end{definition}
        
        Then, $(B_{p,q}^s(\T),  \lVert \cdot \rVert_{B_{p,q}^s})$ is a Banach space when $p$ and $q \geq 1$ and a quasi-Banach space when $p$ or $q <1$. 
        The choice of the system $\bm{\varphi}$ does not change  $B_{p,q}^s(\T)$ as a set and two different systems define equivalent Besov (quasi-)norms~\cite[Theorem 1, Section 3.5.1]{schmeisser87}.
        
    \subsection{Wavelet Frames and Besov Sequence Spaces}
    
    Wavelets can be used to characterize the Besov smoothness of generalized periodic functions from their wavelet coefficients.
    In comparison, this is not possible with Fourier series\footnote{For instance, any square-integrable function $f \in L_2(\T)$ can be transformed into a continuous function by only changing the phase of its Fourier series coefficients~\cite{kahane1977coefficients}.}~\cite{MeyerWaO}. 
    The wavelet frames (or bases) that we consider are Parseval frames of $L_2(\T)$ (see~\cite[Definition 5.1.2]{christensen2016introduction} for a precise definition). Parseval frames share most of the interesting properties of orthonormal bases but can admit some redundancies, and many wavelet systems are actually frames.
        
    A \emph{wavelet frame} is defined by a collection of finite sets $\{X_j\}_{j \geq 0}$ such that 
    \begin{equation}\label{eq:XjC}
    2^j \leq \#X_j \leq C 2^j 
    \end{equation}
    for every $j \geq 1$, where $C\geq 1$ is a constant independent of $j \geq 1$,
    and a collection of functions $\Psi_{j,k} \in L_2(\T)$ with $j \geq 0$ and $k \in X_j$ such that
    the family
    \begin{equation}
      \bm{\Psi} =  \left( \Psi_{j,k} \right)_{j\geq 0,k\in X_j}
    \end{equation}
    forms a Parseval frame of $L_2(\T)$. 
    The parameter $j \geq 0$ plays the role of the scale.
    A typical example is the Haar basis, for which $X_0 = \{0,1\}$ is of size $2$ with $\phi = \Psi_{0,0} = 1$ and $\psi = \Psi_{0,1} = 1_{[0,1/2)} - 1_{[1/2,1)}$ and $X_j = \{0, \ldots , 2^{j-1} \}$ is of size $2^j$ for each $j \geq 1$ with  $\Psi_{j,k} = 2^{j/2} \psi(2^j \cdot - k)$. This corresponds to $C = 1$ in \eqref{eq:XjC}. The Haar system is an orthonormal basis of $L_2(\T)$. 
    
    The \emph{Besov sequence spaces} $b_{p,q}^s$ are indexed in the same way as the wavelets, and they are defined by the norm
    \begin{equation} \label{eq:sequencenorm}
        \lVert a \rVert_{b_{p,q}^s} 
        = 
        \left(   
        \sum_{j \geq 0} 
        2^{j  \left( s - \frac{1}{p} \right) q} 
        \left( \sum_{k \in X_j} | a_{j,k} |^p \right)^{q/p} \right)^{1/q},
    \end{equation}
    with the usual adaptation when $p$ or $q=\infty$. The Besov sequence spaces and the periodic Besov spaces are connected by the following definition.     
        
    \begin{definition} \label{def:admissible}
    Let $0 \leq p_0 < \infty$ and $0 < s_0 \leq \infty$.
     We say that a Parseval frame of the type above is \emph{$(p_0,s_0)$-admissible} if the Besov norm is equivalent to the Besov sequence space norm for all the spaces $B_{p,q}^s(\T)$ with $|s| < s_0$, $p > p_0$, and $q > 0$, in the sense that
     \begin{equation}
        \lVert a \rVert_{b_{p,q}^s} \quad \text{and} \quad \lVert f \rVert_{B_{p,q}^s}
     \end{equation}
     are two equivalent (quasi-)norm where
     \begin{equation}
         \label{eq:relationsequencefunction}
            a_{j,k} = 2^{j/2} \langle f ,\Psi_{j,k} \rangle
     \end{equation}
     for any $j \geq 0$ and $k\in X_j$. 
    \end{definition}
    
   \textit{Remarks.} (i) If a Parseval frame is $(p_0,s_0)$-admissible, then the wavelets are all in $B_{p_0}^{s_0}(\T)$.
    The condition $|s|< s_0$ is useful also for $s <0$: if $ s \leq -s_0$, it may be that the wavelet coefficients between $f \in B_{p}^s(\T)$ and the wavelets $\Psi_{j,k} \in B_{p_0}^{s_0}(\T)$ are not well-defined. 
    In other terms, a $(p_0,s_0)$-admissible Parseval frame can be used to characterize the Besov regularity of periodic generalized functions via their wavelet coefficients
    for parameters $s,p,q$ such that $|s| < s_0$, $p > p_0$, and $q > 0$. 
    
    (ii) For the relation $a_{j,k} = 2^{j/2} \langle f ,\Psi_{j,k} \rangle$ in the equivalence between periodic Besov spaces and Besov sequence spaces, we follow the convention of~\cite{Triebel2008function}. It is possible to chose a different one, as done for instance in~\cite{narcowich2006decomposition}\footnote{In~\cite{narcowich2006decomposition}, the authors define $a_{j,k} = \langle f ,\Psi_{j,k} \rangle$, which requires to change \eqref{eq:sequencenorm}, where $ 2^{j  \left( s - \frac{1}{p} \right) q} $ becomes $ 2^{j  \left( s - \frac{1}{p} + \frac{1}{2} \right) q} $.}.  \\
    
     Admissible orthonormal bases and more general Parseval frames of $L_2(\T)$ are known to exist.
   The admissibility of the periodic Daubechies wavelets considered in \cite{Triebel2008function}
    depend on the smoothness of the wavelet.
    It is particularly interesting to design {$(0,\infty)$}-admissible Parseval frames, since we can use them for any tempered generalized functions $f \in \SpT$ with no restriction on the smoothness (this is not the case for Daubechies wavelets due to their limited smoothness).  This requires in particular that $\Psi_{j,k} \in \ST$ for any $j \geq 0$ and $k \in X_j$. 
     The Parseval frames of  \cite{narcowich2006localized,narcowich2006decomposition} are $(0,\infty)$-admissible. 
    While we could not find a detailed proof in the literature, periodized Meyer wavelets, that form an orthonormal basis of $L_2(\T)$~\cite{offin1993note,jaffard2000lacunary}, should also be $(0,\infty)$-admissible. From now on, we assume that we fix some $(0,\infty)$-admissible Parseval frame $\bm{\Psi}$. Hence, for any $f \in \SpT$ and Besov sequence $a = (a_{j,k})_{j \geq 0, k \in X_j}$ given by \eqref{eq:relationsequencefunction}, we have that
    \begin{equation} \label{eq:equivalencefanda}
        f \in B_{p,q}^s(\T) \Longleftrightarrow a \in b_{p,q}^s
    \end{equation}
    for any $0< p,q \leq \infty$ and $s\in \R$.

        \subsection{Embedding between Besov Spaces}
        
        We summarize embedding results between periodic Besov spaces. In Proposition~\ref{prop:embeddings}, we give sufficient conditions such that a given Besov space is included in another one. In Proposition~\ref{prop:interpol}, we give sufficient conditions to ensure that an intersection of two Besov spaces is included in a third one. This second result relies on interpolation theory.
        
        \begin{proposition} \label{prop:embeddings}
        Let $0 < p_0, p_1, q_0 , q_1 \leq \infty$, $s \in \R$, and $\epsilon > 0$. If $p_0 \leq p_1$, then, we have the topological embeddings:
        \begin{align}
            B_{p_0,q_0}^{s + \frac{1}{p_0} - \frac{1}{p_1}} (\T)  \subseteq B_{p_1,q_1}^{s- \epsilon} (\T) \quad \text{and} \quad 
            B_{p_1,q_1}^{s} (\T)  \subseteq B_{p_0,q_0}^{s- \epsilon} (\T).
        \end{align}
        \end{proposition}

        Proposition~\ref{prop:embeddings} compiles the results of~\cite[Section 3.5.5]{schmeisser87}.


        \begin{proposition} \label{prop:interpol}
        Let $0 < p_0, p_1 \leq \infty$, $0 <  q_0 , q_1 \leq \infty$ and $s_0, s_1 \in \R$. 
        For any $\lambda \in (0,1)$, we set 
        \begin{equation}
            \frac{1}{p} =  \frac{\lambda}{p_0}  +  \frac{1 - \lambda}{p_1}, \quad  s = \lambda s_0 + (1-\lambda) s_1, \quad \text{and} \quad \frac{1}{q} =  \frac{\lambda}{q_0}  +  \frac{1 - \lambda}{q_1}.
        \end{equation} 
        Then, we have that
        \begin{equation} \label{eq:interpol}
             B_{p_0,q_0}^{s_0}(\T) \cap  B_{p_1,q_1}^{s_1}(\T) \subseteq   B_{p,q}^{s}(\T),
        \end{equation}
        where the (quasi-)norm on the intersection is $\lVert \cdot \rVert_{B_{p_0}^{s_0}} + \lVert \cdot \rVert_{B_{p_1}^{s_1}}$, which specifies a {(quasi-)Banach} topology.
        \end{proposition}
        
        
        Proposition~\ref{prop:interpol} is a direct corollary of the theorem presented in~\cite[Section 3.6.2]{schmeisser87} on the complex interpolation of periodic Besov spaces.  Indeed, this result states that the interpolation space between $B_{p_0,q_0}^{s_0}(\T)$ and  $B_{p_1,q_1}^{s_1}(\T)$, which contains the intersection $B_{p_0,q_0}^{s_0}(\T) \cap  B_{p_1,q_1}^{s_1}(\T)$ by definition, is embedded in $B_{p,q}^{s}(\T)$. In Figure~\ref{fig:embed}, we represent the embeddings of Propositions~\ref{prop:embeddings} and \ref{prop:interpol} in $(1/p,s)$-diagrams.

\begin{figure}[h!]  
\centering 
\begin{subfigure}[b]{0.40\textwidth}
\begin{tikzpicture}[x=2cm,y=2cm,scale=0.34]

\fill[green, opacity= 0.4] (0,-1) -- (2,1) -- (3,1) -- (3,-3/2) -- (0,-3/2) -- cycle; 
\draw[green, thick, dashed,->](2,1) --(3,1);
\draw[green, thick, dashed,- ](0,-1) -- (2,1);
\draw[green, thick,->](0,-1) -- (0,-3/2);

\fill[red, opacity= 0.5] (0,1) -- (2,1) -- (3,2) -- (0,2) -- cycle; 
\draw[red, thick, dashed,->](2,1) --(3,2);
\draw[red, thick, dashed,- ](2,1) -- (0,1);
\draw[red, thick,->](0,1) -- (0,2);

\draw[thick, ->] (0,0)--(3,0) node[circle,right] {$\frac{1}{p}$} ;
\draw[thick, ->] (0,-3/2)--(0,2) node[circle,above] {$s$} ;

\draw[ thick,color=black]
(-0.05,1) -- (0.05,1)  node[black,left] { $s_0$};

\draw[ thick,color=black]
(2,-0.05) -- (2,0.05)  node[black,below] { $\frac{1}{p_0}$};

\end{tikzpicture}
\end{subfigure}
\begin{subfigure}[b]{0.40\textwidth}
\begin{tikzpicture}[x=2cm,y=2cm,scale=0.34]

\fill[green, opacity= 0.4] (0.6,-3/2) -- (0.6,0.6) -- (2.3,1.3) -- (2.3,-3/2) -- cycle; 
\draw[green, thick, dashed,- ](0.6,0.6) -- (2.3,1.3);


\draw[thick, ->] (0,0)--(3,0) node[circle,right] {$\frac{1}{p}$} ;
\draw[thick, ->] (0,-3/2)--(0,2) node[circle,above] {$s$} ;

\draw[ thick,color=black]
(-0.05,1.3) -- (0.05,1.3)  node[black,left] { $s_0$};

\draw[ thick,color=black]
(-0.05,0.6) -- (0.05,0.6)  node[black,left] { $s_1$};

\draw[ thick,color=black]
(2.3,-0.05) -- (2.3,0.05)  node[black,below] { $\frac{1}{p_0}$};

\draw[ thick,color=black]
(0.6,-0.05) -- (0.6,0.05)  node[black,below] { $\frac{1}{p_1}$};

\end{tikzpicture}
\end{subfigure}

\caption{Representation of the embeddings between Besov spaces. \emph{On the left:} If $f \in B_{p_0,q}^{s_0}(\T)$, then $f$ is in every Besov space that is in the lower shaded green regions. Conversely, if $f \notin B_{p_0}^{s_0}(\T)$, then $f$ is in none of the Besov spaces of the upper shaded red  region. The slope of the diagonal embedding line is $1$. \emph{On the right:} If $f \in B_{p_0,q_0}^{s_0}(\T) \cap B_{p_1,q_1}^{s_1}(\T)$, then $f$ is in every Besov space that is in the lower shaded green region.}
\label{fig:embed}
\end{figure}
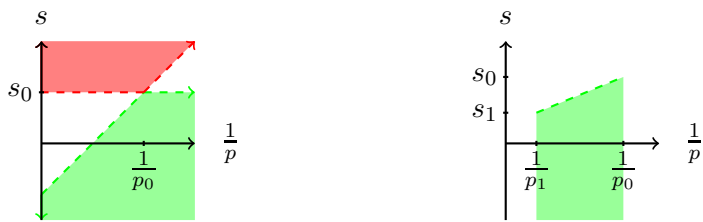

    \section{The Class of Critical Smoothness Functions}
    \label{sec:mainresult}
    
    Our goal in this section is to  prove the main result of this paper, given by Theorem \ref{theo:main} below. Note that the properties of a critical smoothness function are better expressed in terms of $\frac{1}{p} \mapsto s_f(p)$. This is reminiscent to the fact that the Besov regularity of a function is well-captured in $(1/p,s)$-diagrams (see Figure~\ref{fig:embed}) and is consistent with the examples encountered in Section \ref{sec:related}. 
    
    \begin{theorem} \label{theo:main}
    Let $f \in \SpT \backslash \ST$. Then,
    the function $\frac{1}{p} \mapsto s_f(p)$ is an increasing, concave, and $1$-Lipschitz function from $[0,\infty)$ to $\R$.
    
    Moreover, for any function $s : (0,\infty] \rightarrow \R$ such that $\frac{1}{p}\mapsto s(p)$ is increasing, concave, and $1$-Lipschitz, there exists a periodic generalized function $f \in \SpT \backslash \ST$ such that $s_f(p) = s(p)$.
    \end{theorem}
    
    First of all, we observe that   the critical smoothness functions of Section \ref{sec:related} all satisfy the conditions of Theorem \ref{theo:main}. We exclude the case $f \in \ST$ thanks to Proposition \ref{prop:removeS}, infinitely smooth functions being the only ones for which the critical smoothness can be infinite. 
    
    We separate the proof of Theorem \ref{theo:main} in two parts. First, we use embeddings and interpolation theory of Besov spaces to demonstrate the properties of $s_f(p)$ in Section \ref{sec:properties}, where we also study the main properties of critical smoothness functions. Second, we demonstrate that the functions $s$ satisfying these properties are the critical smoothness functions of some $f \in \SpT$ in Section \ref{sec:constructf}.
    
    \subsection{Properties of $s_f(p)$}
    \label{sec:properties}
    
    \begin{proposition}[First part of Theorem \ref{theo:main} restated]
    \label{prop:sfp}
    Let $f \in \SpT$. Then, $s_f(p) = \infty$ for any $p>0$ if $f \in \ST$. Otherwise, 
    the function $\frac{1}{p} \mapsto s_f(p)$ is an increasing, concave, and $1$-Lipschitz function from $[0,\infty)$ to $\R$. 
    \end{proposition}
    
    \begin{proof}
    The case $f \in \ST$ is covered by Proposition \ref{prop:removeS}, which also implies that $s_f(p) < \infty$ for any $f \notin \ST$ and any $0 < p \leq \infty$. We now assume that $f \notin \ST$.  Due to Proposition \ref{prop:removeq}, we can focus on Besov spaces of the form $B_{p,p}^s(\T)$, that we denote by $B_p^s(\T)$ to simplify, for the characterization of $s_f$. We fix $0 < p_1 \leq p_2 \leq \infty$. \\

    \textit{Monotonicity.} Let $s < s_f(p_2)$ so that $f \in B_{p_2}^{s}(\T)$. For any $\epsilon > 0$, we have the embedding $ B_{p_2}^{s}(\T) \subseteq B_{p_1}^{s-\epsilon}(\T)$, from which we deduce that $s_f(p_1) \geq s - \epsilon$. This is valid for any  $s < s_f(p_2)$ and $\epsilon > 0$, hence $s_f(p_1) \geq s_f(p_2)$. In other terms, $\frac{1}{p} \mapsto s_f(p)$ is increasing. \\
    
    \textit{Lipschitzness.} 
    The monotonicity implies that $
    0 \leq \frac{s_f(p_1) - s_f(p_2)}{\frac{1}{p_1} - \frac{1}{p_2}}. $
    Let $s < s_f(p_1)$ so that $f \in B_{p_1}^{s}(\T)$. For any $\epsilon > 0$, we have 
    \begin{equation}
    B_{p_1}^s(\T) \subseteq B_{p_2}^{s - \left(\frac{1}{p_1} - \frac{1}{p_2} \right)- \epsilon}(\T)
    \end{equation}
    due to Proposition \ref{prop:embeddings}. Hence, $s_f(p_2) \geq s - \left(\frac{1}{p_1} - \frac{1}{p_2} \right)- \epsilon$. This is true for any $\epsilon > 0$ and $s < s_f(p_1)$, therefore $s_f(p_2) \geq s_f(p_1)-\left(\frac{1}{p_1} - \frac{1}{p_2} \right)$, \emph{i.e.}, 
    \begin{equation}\label{eq:slope1}
    0 \leq \frac{s_f(p_1) - s_f(p_2)}{\frac{1}{p_1} - \frac{1}{p_2}} \leq 1.
    \end{equation}  

    \textit{Concavity.}
    Let $\lambda \in [0,1]$ and define $p_1 \leq p \leq p_2$ such that $\frac{1}{p} = \frac{\lambda}{p_1} + \frac{1- \lambda}{p_2}$. Fix $\epsilon > 0$ and $s < \lambda s_f(p_1) + (1-\lambda) s_f(p_2) - \epsilon$. By definition, $f \in B_{p_1}^{s_f(p_1)-\epsilon}(\T) \cap B_{p_2}^{s_f(p_2)-\epsilon}(\T)$, and by Proposition \ref{prop:interpol}, $ B_{p_1}^{s_f(p_1)-\epsilon}(\T) \cap B_{p_2}^{s_f(p_2)-\epsilon}(\T) \subseteq B_{p}^{s}(\T)$. 
    This shows that $f \in B_{p}^{s}(\T)$ hence $s_f(p) \geq s$. Choosing $\epsilon \rightarrow 0$ and $s \rightarrow \lambda s_f(p_1) + (1-\lambda) s_f(p_2)$ finally gives $s_f(p) \geq  \lambda s_f(p_1) + (1-\lambda) s_f(p_2)$; \emph{i.e.}, $\frac{1}{p} \mapsto s_f(p)$ is concave.
    \end{proof}
    
    We state some additional simple and useful properties of the critical smoothness function.
    For $\gamma \in \R$, we define the \textit{Sobolev operator}\footnote{See~\cite{fageot2020tv}, especially Section 3.0 for a discussion on periodic operators and Section 5.1 where the Sobolev operator is discussed.} $\Lop_\gamma = (\mathrm{Id} - \Delta)^{\gamma/2}$ acting on periodic functions as
    \begin{equation}
        (\mathrm{Id} - \Delta)^{\gamma/2} f = \sum_{k \in \Z} (1 + k^2)^{\gamma/2} \widehat{f}[k] e_k. 
    \end{equation}
    The effect of $\Lop_\gamma$ is to reduce the smooothness of a factor $\gamma$ (when $\gamma < 0$, it is more correct to say that the smoothness is increased by $-\gamma$). 
    
    \begin{proposition} \label{prop:simplepropsf}
    Let $f \in \SpT \backslash \ST$. 
    
            \begin{itemize}[label=\raisebox{0.30ex}{\tiny$\bullet$}] 
        
            \item For any $\gamma \in\R$, we have that, for all $0<p \leq \infty$,
            \begin{equation} \label{eq:withLopgamma}
                 s_{\Lop_\gamma f}(p)= s_f(p) - \gamma. 
            \end{equation}
           \item The function $\frac{1}{p} \mapsto s_f(p)$ is left and right differentiable at any point. 

    We denote by $\partial_+s_f(p_0)$ ($\partial_- s_f(p_0)$, resp.) the right (left, resp.) differential of $\frac{1}{p} \mapsto s_f(p)$ in $p_0 \in (0,\infty]$ ($p_0 \in (0,\infty)$, resp.). 
        \item For any $0 <p<\infty$, we have that $\partial_-s_f(p) \geq \partial_+ s_f(p)$. 
        \item There are at most countably many points where the function $\frac{1}{p} \mapsto s_f(p)$ is not differentiable.
         \item If $\partial_-s_f(p_0)= 1$, then for any  $p \geq p_0$, we have 
        $s_f(p) = s_f(p_0) + \frac{1}{p} - \frac{1}{p_0}$. 
        \item If $\partial_+s_f(p_0)= 0$, then for any  $p \leq p_0$, we have 
        $s_f(p) = s_f(p_0)$. 
        \item If $\partial_-s_f(p_0)= 1$ and $\partial_+s_f(p_0)= 0$, then for all $0< p \leq \infty$, we have
        \begin{equation} \label{eq:only01slopes}
            s_f(p) = s_f(p_0) - \frac{1}{p_0} + \frac{1}{\max(p,p_0)} = s_f(\infty) + \frac{1}{\max(p,p_0)}.
        \end{equation}
    \end{itemize}
    \end{proposition}
    
    \begin{proof}
        \begin{itemize}[label=\raisebox{0.30ex}{\tiny$\bullet$}]
        \item The operator $\Lop_\gamma$ is a continuous bijection from $B_{p+\gamma}^s(\T)$ to $B_p^{s}$  for any $\gamma \in \R$, $s \in \R$, and $0<p \leq \infty$~\cite[Section 2.3.8]{Triebel2010theory}\footnote{Hans Triebel details the properties of $\Lop_\gamma$ (denoted by $I_\gamma$) for tempered generalized functions, but the result is easily adapted to periodic generalized functions.}. This directly implies the desired relation on the critical smoothness functions of $f$ and $\Lop_\gamma f$. 
    
   \item   By concavity, $\frac{1}{p} \mapsto \frac{ s_f(p) - s_f(p_0) }{\frac{1}{p} - \frac{1}{p_0}}$ decreases when $\frac{1}{p}$ decreases to $\frac{1}{p_0}$. It is moreover bounded below by $0$, and therefore admits a limit when $\frac{1}{p} \rightarrow \left( \frac{1}{p_0}\right)^+$, which is the right differentiable at $\frac{1}{p_0}$. 
    We proceed similarly for $p_0 \in (0,\infty)$ to show the left differentiability (we exclude $p_0 = \infty$, corresponding to $\frac{1}{p_0} = 0$ in this case). 
    
    \item The concavity directly implies that the left derivative is bigger than the right derivative at any point.
    
    \item The function $\frac{1}{p} \mapsto \partial_+ s_f(p)$ is decreasing (because $\frac{1}{p} \mapsto s_f(p)$ is concave). According to the Darbou--Froda theorem~\cite[Theorem 4.30 p. 96]{rudin1964principles}, it is therefore continuous at every point except possibly countably many. The same holds for $\frac{1}{p} \mapsto \partial_- s_f(p)$, therefore these two functions are equal except possibly at finitely many points, as expected.
    
  \item   Assume that $\partial_+s_f(p_0)= 1$, then the slopes of $\frac{1}{p} \mapsto s_f(p)$ are smaller than $1$ ($1$-Lipschitzness) and greater than $\partial_+s_f(p_0) = 1$ for $p \geq p_0$ by concavity. Hence, we have that $\frac{s_f(p) - s_f(p_0)}{\frac{1}{p} - \frac{1}{p_0}} = 1$ for any $p > p_0$ and the result follow. A similar argument gives that $s_f(p) = s_f(p_0)$ for $p \leq p_0$ if $\partial_- s_f(p_0)=0$. 
    
    \item Finally, \eqref{eq:only01slopes} follows from the two previous cases.
        \end{itemize}
    \end{proof}

    \textit{Remark.} The critical smoothness function \eqref{eq:only01slopes} can be compared with the ones presented in Section \ref{sec:related}. The form of the encountered critical smoothness functions can all be put in this form (possibly with $p_0 = \infty$). We shall see in the next section that other critical smoothness functions can be created.

    \begin{proposition} \label{prop:someproperties}
    Let $f, g \in \SpT$ and $\lambda \neq 0$.
        \begin{itemize}[label=\raisebox{0.30ex}{\tiny$\bullet$}]
        \item For any $0 < p \leq \infty$, $s_{\lambda f}(p) = s_f(p)$.
        \item If $f$ and $g$ are such that $s_f(p) \neq s_g(p)$ for some $0 < p \leq \infty$, then $s_{f+g}(p) = \min(s_f(p) , s_g(p))$.
        \item If $f$ and $g$ are such that  $s_f(p) = s_g(p)$ 
        for at most countably many  $0 < p \leq \infty$
        , then $s_{f+g} = \min(s_f, s_g)$.
    \end{itemize}
    \end{proposition}
    \begin{proof}
    The first point is obvious because $B_{p}^s(\T)$ is linear. For the second, assume for instance that $s_f(p) < s_g(p)$ and let $s <   s_f(p)$. Then $f$ and $g$ are in the linear space $B_p^s(\T)$, so is the sum $f+g$. This shows that $s_{f+g}(p) \geq s_f(p)$. Assume now that $s_f(p) < s < s_g(p)$, then $f \in B_p^s(\T)$ and $g \notin B_p^s(\T)$. Therefore, $f+g \notin B_p^s(\T)$ (otherwise, $g= (g+f) - f$ would be in $B_p^s(\T)$, what we excluded). Hence, $s_{f+g}(p) \leq s_f(p)$. Finally, we have shown that $s_{f+g}(p) = s_f(p) =  \min(s_f(p) , s_g(p))$. For the last point, we have that $s_{f+g}(p) =  \min(s_f(p) , s_g(p))$ for any $p$ such that $s_f(p) \neq s_g(p)$. In particular, the  functions $s_{f+g}$ and $\min(s_f , s_g)$ coincide, expect possibly over a countable set. The two functions being continuous, we deduce that they are equal. 
    \end{proof}

    \subsection{Characterizing all the Critical Smoothness Functions}
    \label{sec:constructf}
    
    So far, we only encountered critical smoothness functions with left or right differential being $0$ or $1$. We start with the construction of a a function with affine critical smoothness function with any possible slope between $0$ and $1$. 
    
    \begin{proposition} \label{prop:sfplinear}
    Let $s_0 \in \R$ and $0 \leq \alpha_0 \leq 1$. There exists a function $f \in \SpT$ such that 
    \begin{equation}\label{eq:sfpalpha}
        s_f(p) = s_0 + \frac{\alpha_0}{p}, \quad \forall 0 < p \leq \infty. 
    \end{equation}
    \end{proposition}
    
    This result is, strictly speaking, not new. We have for instance seen in Section \ref{sec:related} that~\cite{bochkina2013besov} considers stochastic processes whose critical smoothness function satisfies~\eqref{eq:sfpalpha}. Thereafter, we provide an elementary proof of Proposition~\ref{prop:sfplinear} both for the sake of completeness and because we want to consider any possible values of the parameters, including $s_0 < 0$ and $p<1$ (which are more rarely considered in the literature). 
    
    \begin{proof}[Proof of Proposition~\ref{prop:sfplinear}]
    The parameter $\alpha_0$ being fixed, it suffices to construct a function such that \eqref{eq:sfpalpha} holds for some $s_0$ to ensure that one can construct $f$ for any other value of $s_0$, thanks to \eqref{eq:withLopgamma}. Note moreover that we may restrict our attention to $p < \infty$, the case $p = \infty$ being deduced by continuity ($\frac{1}{p} \mapsto s_f(p)$ is continuous at $0$, corresponding to $p = \infty$). 
    
    The cases $\alpha_0 = 0$ and $\alpha_0 = 1$ are already covered by known examples. For $\alpha_0 = 0$, we know that $s_B(p) = \frac{1}{2}$ (a.s.) for the Brownian motion $B$. For $\alpha_0 = 1$, we know that $s_{\delta}(p) = \frac{1}{p} - 1$. One can therefore assume that $\alpha_0 \in(0,1)$.
    
    Let $K > 1$ such that $\alpha_0 = 1 - \frac{1}{K}$. We construct the function $f$ from its wavelet coefficients. At a given scale $j \geq 0$, we impose that, among the $2^j$ wavelet coefficients of $f$, $2^{\lfloor j/K \rfloor} \leq 2^j$ are equal to $2^{j/2}$ while the others are $0$. In particular, this implies that, for any scale $j \geq 0$, 
    \begin{equation} \label{eq:upperlowerf}
       \frac{1}{2}
       2^{j \left( \frac{p}{2} + \frac{1}{K}\right) } \leq  \sum_{k=0}^{2^{j}-1} |\langle f , \psi_{j,k} \rangle|^p = 2^{j \frac{p}{2}} 2^{ \lfloor j/K \rfloor} \leq 2 \cdot 2^{j \left( \frac{p}{2} + \frac{1}{K}\right) } . 
    \end{equation}
    In particular, the Besov (quasi-)norm of $f$, given by \eqref{eq:besovnorm}, satisfies
    \begin{equation}
      \frac{1}{2} \sum_{j\geq 0} 2^{j \left( s p - 1 + \frac{1}{K} \right) }  \leq \lVert f \rVert_{B_p^s}^p   \leq 2 \sum_{j\geq 0} 2^{j \left( s p - 1 + \frac{1}{K} \right)}. 
    \end{equation}
    This means that $f \in B_p^s(\T)$ if and only if $s p - 1   + \frac{1}{K} < 0$, \emph{i.e.}, if and only if $s < \frac{1}{p} \left( 1 - \frac{1}{K} \right) = \frac{\alpha_0}{p}$, and we have, in this case,
    \begin{equation}\label{eq:givemeanorm}
     \frac{1}{2} \frac{1}{1 - 2^{sp - \alpha_0}} \leq \lVert f \rVert_{B_p^s}^p   \leq   \frac{2}{1 - 2^{sp - \alpha_0}}.
    \end{equation}
    This shows that $s_f(p) = \frac{\alpha_0}{p}$, as expected. 
    \end{proof}
    
    The next proposition is dedicated to the construction of a generalized function $f$ with a given critical smoothness function. We demonstrate the conditions of Proposition \ref{prop:sfp} are sufficient to ensures that $p \mapsto s(p)$ is the critical smoothness function of some generalized function. This covers functions such as
    \begin{equation}
        \frac{1}{p} \mapsto \log \left( \frac{1}{p} + 1 \right), \quad \text{and} \quad \frac{1}{p} \mapsto    \left( \frac{1}{p}+1  \right)^{\beta},
    \end{equation}
    for any $\beta \in (0,1)$, which are clearly $1$-Lipschitz, concave, and increasing with respect to $\frac{1}{p}$.
    
    \begin{proposition}[Second part of Theorem \ref{theo:main} restated]
    \label{prop:theomain2}
For any function $s : (0,\infty] \rightarrow \R$ such that $\frac{1}{p}\mapsto s(p)$ is increasing, concave, and $1$-Lipschitz, there exists a periodic generalized function $f \in \SpT$ such that $s_f(p) = s(p)$.
    \end{proposition}
    
    \begin{proof}
    Using the isometric isomorphism between periodic Besov spaces and sequence spaces,especially~\eqref{eq:equivalencefanda}, it suffices to construct a Besov sequence $a$ such that $s_a(p) = s(p)$ for any $0< p \leq \infty$. We will therefore directly work on Besov sequences. Note moreover that we may restrict to $p < \infty$, the case $p = \infty$ always following by continuity of the critical smoothness functions at $\frac{1}{p} = 0$. 
    
    We consider sequences $(s_n)_{n\geq 1}$ and $(\alpha_n)_{n\geq 1}$ such that (i) for every $n \geq 1$ and $0<p \leq \infty$, $s_n \in \R$, $\alpha_n \in [0,1]$,  
        \begin{equation}
        \label{eq:snalphansp}
       s_n + \frac{\alpha_n}{p} > s(p) \quad \text{and} \quad  s(p) = \inf_{n\geq 1} \left( s_n + \frac{\alpha_n}{p} \right),
       \end{equation}
     and (ii) requiring that $1 - 2^{- \epsilon_n} = \frac{1}{n}$, where $\epsilon_n > 0$ is given by
        \begin{equation} \label{eq:epsn}
            \epsilon_n= \inf_{0 < p \leq \infty} s_{a_n}(p) - s(p) = \inf_{0 < p \leq \infty} s_n + \frac{\alpha_n}{p} - s(p).
        \end{equation}
       This gives constraints on the numbers $\alpha_n$ and $s_n$, that can be constructed as follows. First, we consider a sequence $(p_n)_{n\geq 1}$ which is dense in $(0,\infty)$ and such that no $\frac{1}{p_n}$ is a point of non-differentiability of $\frac{1}{p} \mapsto s(p)$ (this is possible because $\frac{1}{p} \mapsto s(p)$ has at most countably many points of non-differentiability, as shown in the proof of Proposition~\ref{prop:simplepropsf}).
        Then, $a_n$ is chosen as the slope of $\frac{1}{p} \mapsto s_f(p)$ at $p = p_n$ an $s_n$ is such that $s_n + \frac{a_n}{p_n} = s_f(p_n) + \epsilon_n$.
        Then, using that $\frac{1}{p}\mapsto s(p)$ is concave, it is below its tangents and therefore the right relation in \eqref{eq:snalphansp} holds for any $0< p \leq \infty$. Finally, the left relation in \eqref{eq:snalphansp} follows from the fact that $\epsilon_n \rightarrow 0$.
        
        The sequences $(s_n)_{n\geq 1}$ and $(\alpha_n)_{n \geq 1}$ being constructed, we define for each $n \geq 1$ the Besov sequence $a_n$ as in the proof of Proposition \ref{prop:sfplinear} (via the isometric isomorphism  between periodic and sequence Besov spaces). In particular, due to \eqref{eq:givemeanorm}, we have that 
        \begin{equation}\label{eq:willbeuseful}
       \frac{1}{2}  \frac{1}{( 1 - 2^{p ( s - s_{a_n}(p)  ) })} 
       \leq  \lVert a_n \rVert_{b_{p}^s}^p 
       \leq
         \frac{2}{( 1 - 2^{p ( s - s_{a_n}(p)  ) })} . 
        \end{equation}
        Moreover, using that $\epsilon_n \leq s_{a_n}(p) - s(p)$, we deduce from \eqref{eq:willbeuseful} that 
        \begin{equation}
            \label{eq:usefulalso}
            \lVert a_n \rVert_{b_{p}^s}^p  \leq  \frac{2}{( 1 - 2^{- p \epsilon_n })} = \frac{2}{ 1 - \left( 1 - \frac{1}{n} \right)^p } \underset{n \rightarrow \infty}{\sim} \frac{2n}{p}.
        \end{equation}
        
        We will show thereafter that $a = \sum_{n\geq 1}\frac{1}{2^n} a_n$ is a Besov sequence such that $s_a(p) = s(p) = \inf_{n\geq 1} s_{a_n}(p)$ for every $0<p < \infty$. 
  
     (i) We first show that $s_a(p) \geq \inf_{n\geq 1} s_{a_n}(p)$ by demonstrating that $\lVert a \rVert_{b_p^{s(p)}}^p< \infty$, where we recall that $s(p)$ satisfies \eqref{eq:snalphansp}. We distinguish the proof for $p \geq 1$ and $0<p<1$. Assume therefore that $p\geq 1$. We will show that 
    \begin{equation} \label{eq:showsomething}
        \sum_{n\geq 1} \frac{1}{2^n} \lVert a_n \rVert_{b_p^{s(p)}}  <  \infty. 
    \end{equation}
    In particular, this implies that the series is absolute convergent in the Banach space $b_p^{s(p)}$, and therefore converges in this space to a limit $a \in b_p^{s(p)}$~\cite[p. 11]{fabian2011banach}, hence $a$ is a well-defined Banach sequence.
    We now prove \eqref{eq:showsomething}. We have
    \begin{equation} \label{eq:pifpaf}
        \sum_{n\geq 1} \frac{1}{2^n} \lVert a_n \rVert_{b_p^s} 
        \leq
        \sum_{n\geq 1} \frac{1}{2^n} \left( \frac{2}{1 - \left( 1 - \frac{1}{n} \right)^p} \right)^{1/p}
        =         \sum_{n\geq 1} \frac{1}{2^n} \left( \frac{2}{1 - \left( 1 - \frac{1}{n} \right)^p} \right)^{1/p},
    \end{equation}
    where we used \eqref{eq:usefulalso}. The summand in \eqref{eq:pifpaf} is asymptotically equivalent to $\frac{1}{2^n} \left( \frac{ 2 n}{p}\right)^{1/p}$, which clearly defines a convergent series. This proves that $ \sum_{n\geq 1} \frac{1}{2^n} \lVert a_n \rVert_{b_p^s}  < \infty$.

    Assume now that $0 < p < 1$. It suffices to show that $\sum_{n\geq 1} \frac{1}{2^n}^p \lVert a_n \rVert_{b_p^{s(p)}}^p < \infty$ to deduce that $a$ lies in the quasi-Banach (and therefore complete) space $b_p^{s(p)}$ (note that we consider a quantity different from \eqref{eq:showsomething}). 
    Then, \eqref{eq:usefulalso} implies that 
     \begin{equation} \label{eq:ploufplouf}
        \sum_{n\geq 1} \left(\frac{1}{2^n}\right)^p \lVert a_n \rVert_{b_p^{s(p)}}^p \leq \sum_{n\geq 1}
        \frac{1}{2^{np}} \frac{ 2 }{ 1 - \left( 1 - \frac{1}{n} \right)^p}.
    \end{equation}   
    The summand in~\eqref{eq:ploufplouf} is asymptotically equivalent to $\frac{2n}{p 2^{np}}$, which is the term of a convergent series, hence $ \sum_{n\geq 1} \frac{1}{2^{np}} \lVert a_n \rVert_{b_p^{s(p)}}^p  < \infty$, 
 and therefore $s_a(p) \geq s(p)$. 
 
     (ii) We then show that $s_a(p) \leq \inf_{n\geq 1} s_{a_n}(p)$. Fix $s > \inf_{n\geq 1} s_{a_n}(p)$, so that there exists $n_0 \geq 1$ such that $\lVert a_{n_0} \rVert_{b_p^s} = \infty$. Then, for each $j \geq 0$, using that all the sequence coefficients are non-negative, we have that $\lVert \sum_{n\geq 1} \frac{1}{2^n} a_n^j \rVert_p^p \geq \lVert b_{n_0} a_{n_0}^j \rVert_p^p$. Hence, 
    \begin{equation}
        \lVert a \rVert_{b_p^s}^p = \sum_{j \geq 0} 2^{j (sp -1)} \left\lVert \sum_{n\geq 1} \frac{1}{2^n} a_n^j \right \rVert_p^p  
        \geq \sum_{j \geq 0} 2^{j (sp -1)} \left\lVert  b_{n_0} a_{n_0}^j  \right \rVert_p^p  
        = b_{n_0}^p \lVert a_{n_0} \rVert_{b_p^s}^p = \infty,
    \end{equation}
    which implies that $a \notin b_p^s$. This is true for any $s > \inf_{n\geq 1} s_{a_n}(p)$, therefore $s_a(p) \leq \inf_{n\geq 1} s_{a_n}(p)$.
    Finally, we have shown that 
    \begin{equation}
        s_a(p) = \inf_{n \geq 1} s_{a_n}(p) = s(p),
    \end{equation}
    completing the proof.

    \end{proof}

\textit{Remark.} The proof of Proposition \ref{prop:theomain2} is based on the idea that an infinite sum $a = \sum_{n\geq 1} \frac{a_n}{2^n} $ has excellent chances to satisfy $s_a(p) = \inf_{n\geq 1} s_{a_n}(p)$. This is the infinite generalization of the third point of Proposition~\ref{prop:someproperties} and the main difficulty is to correctly normalize the sum (via the $\frac{1}{2^n}$ and the right choice for the $a_n$ sequences) so that the Besov sequence norms are finite. The critical point is to ensure that the $a_n$ are constructed such that we control the Besov sequence norms for any $p > 0$, the values $p \rightarrow 0$ being the more challenging ones. 
It is worth noting that our proof is \textit{constructive}, in the sense that the $a_n$ and therefore $a$ are concrete Besov sequences.

\section{Critical Smoothness and Compressibility}
\label{sec:compressibility}

The Besov regularity of a function is intimately linked to its best approximation in wavelet bases~\cite{Devore1998Nterm}. We revisit this fact by connecting the \textit{wavelet compressibility} of a generalized function $f$, introduced in~\cite[Definition 13]{Fageot2017nterm} (see Definition \ref{def:compress} below), to its critical smoothness function. 

    \subsection{Compressibility in Wavelet Bases}
    
    In this section, we fix $0<p_0 \leq \infty$ and $s_0 \in \R$. 
    We only consider Besov spaces with $q=p$, that we denote by $B_p^s(\T) = B_{p,p}^s(\T)$ thereafter.
    
    Assume that we have a orthonormal basis $(\psi_{\lambda})_{\lambda \in \Lambda}$ of $L_2(\T)$, where $\Lambda$ is a infinite set of indices. When dealing with functions in $B_{p_0}^{s_0}(\T)$, we will always consider that the wavelet bases is $(p_0,s_0)$-admissible in the sense of Definition~\ref{def:admissible}.
    
    For any $f \in \SpT$, we have  that $f = \sum_{\lambda \in \N} \langle f , \psi_{\lambda} \rangle \psi_{\lambda}$, where the convergence holds in $\SpT$. 
   For $f \in \SpT$ and $\Lambda_N \subset \Lambda$ such that $|\Lambda_N| = N$, we define $\Sigma_{\Lambda_N}(f) = \sum_{\lambda \in \Lambda_N} \langle f , \psi_{\lambda} \rangle \psi_{\lambda}$. 
   
   \begin{proposition} \label{prop:trucdeconv}
  Let $0<p_0\leq \infty$, $s_0\in \R$, and  $f \in \SpT$ be such that $s_f(p_0) > s_0$. Then, $f \in B_{p_0}^{s_0}(\T)$ and the convergence $f = \sum_{\lambda \in \N} \langle f , \psi_{\lambda} \rangle \psi_{\lambda}$ holds in $B_{p_0}^{s_0}(\T)$. 
   \end{proposition}
   
   \begin{proof}
   We have $f \in B_{p_0}^{s_0}(\T)$ by definition of $s_f(p_0)>s_0$. 
   The wavelet basis is $(p_0,s_0)$-admissible. This means that the wavelet approximation of $f$ until the scale $J \geq 1$ goes to $f$ in $B_{p_0}^{s_0}(\T)$ when $J \rightarrow \infty$, which is precisely a reformulation of Proposition~\ref{prop:trucdeconv}.
   \end{proof} 
   
\begin{definition}\label{de:bestnterm}
 Let $0<p_0\leq \infty$, $s_0\in \R$, and  $f \in \SpT$ be such that $s_f(p_0) > s_0$.
We define the \emph{best $N$-term approximation} of $f$ as 
   \begin{equation}
       \Sigma^{p_0,s_0}_N (f) = \Sigma_{\Lambda^{p_0,s_0}_N(f)} (f) = \sum_{\lambda \in \Lambda^{p_0,s_0}_N(f)} \langle f ,\psi_{\lambda} \rangle \psi_{\lambda},
   \end{equation}
   where $\Lambda^{p_0,s_0}_N(f)$ is such that\footnote{In general, the set $\Lambda^{p_0,s_0}_N(f)$ is not unique, but choosing different optimal sets will have no impact on the sequel.}
   \begin{equation}
       \Lambda^{p_0,s_0}_N(f) \in \underset{|\Lambda_N| = N}{\arg \min} \ \lVert f - \sigma_{\Lambda_N} (f) \rVert_{B_{p_0}^{s_0}}.
   \end{equation}
     The \emph{best $N$-term approximation error} is then
  \begin{equation}
      \sigma_N^{p_0,s_0}(f) = \lVert f - \Sigma_N^{p_0,s_0}(f) \rVert_{B_{p_0}^{s_0}}. 
  \end{equation}
  \end{definition}
  
\begin{corollary} \label{coro:conv}
     Let $0<p_0\leq \infty$, $s_0\in \R$, and  $f \in \SpT$ be such that $s_f(p_0) > s_0$. Then, $ \sigma_N^{p_0,s_0}(f) \rightarrow 0$ when $N \rightarrow \infty$. 
\end{corollary}
   
   \begin{proof}
  According to Proposition~\ref{prop:trucdeconv}, we have the convergence of the series  $f = \sum_{\lambda \in \N} \langle f , \psi_{\lambda} \rangle \psi_{\lambda}$ in $B_{p_0}^{s_0}(\T)$, which directly implies Corollary \ref{coro:conv}.
   \end{proof}
   
    The speed of decay of $\sigma_N^{p_0,s_0}(f)$ towards $0$ measures the compressibility, in the Besov space $B_{p_0}^{s_0}(\T)$, of $f$. We are typically interested by polynomial decays, \emph{i.e.}, when $\sigma_N^{p_0,s_0}(f)$ roughly speaking  behaves as $1 / N^{\kappa}$ for some $\kappa > 0$. This leads to the following definition. 

    \begin{definition} \label{def:compress}
    Let $f \in \SpT$ such that $s_f(p_0) > s_0$. The {$(p_0,s_0)$-compressibility} is given by 
    \begin{equation}
        \kappa^{p_0,s_0} (f) = \sup \left\{ \kappa \geq 0, \quad  \frac{\sigma_N^{p_0,s_0}(f)}{N^\kappa} \underset{N\rightarrow \infty}{\longrightarrow} 0 \right\}.
    \end{equation}
    \end{definition}
    The compressibility $\kappa^{p_0,s_0} (f)$ is well-defined (as the supremum is taken over a non-empty set that contains $\kappa = 0$ due to Corollary~\ref{coro:conv}) and takes value in $[0, \infty]$. The case $ \kappa^{p_0,s_0} (f) = \infty$ means that the approximation error decay is super-polynomial, the function $f$ hence being highly compressible. 
    Using known results on Besov spaces and approximation theory, especially the link between Besov sequence spaces and non-linear approximation~\cite{Garrigos2004sharp}, we can deduce the following result. 
    
    \begin{proposition} \label{prop:giveskappa}
    Let $f \in \SpT$, $0 < p_0 \leq \infty$, and $s_0\in \R$ be such that $s_{p_0}(f) > s_0$. Then, the two following scenarios hold. 
   
   (i) Assume that there exists $0 < p(f)\leq p_0$ such that,
    \begin{align}
    f  \in B_{p  }^{\frac{1}{p} - \frac{1}{p_0}+ s_0 }(\T) , \quad &\forall p > p(f), \text{ and }  \label{eq:fin}\\
     f \notin B_{p}^{\frac{1}{p} - \frac{1}{p_0} + s_0}(\T) , \quad &\forall p < p(f). \label{eq:fnotin}
    \end{align}
    Then, the value $p(f)$ is unique and we have
    \begin{equation} \label{eq:kappafini}
        \kappa^{p_0,s_0}(f) = \frac{1}{p(f)} - \frac{1}{p_0}.
    \end{equation}
    
    (ii) Assume that, for any $p \geq p_0$, $f \in B_{p}^{\frac{1}{p}- \frac{1}{p_0} + s_0}(\T)$. Then, 
    \begin{equation} \label{eq:kappainf}
        \kappa^{p_0,s_0}(f) =  \infty.
    \end{equation}
    \end{proposition}

    \textit{Remark.} Proposition \ref{prop:giveskappa} formalizes the fundamental idea that the $(p_0,s_0)$-compressibility is captured by the belonging of $f$ into Besov spaces of the form $B_{p}^{\frac{1}{p} - \frac{1}{p_0}+ s_0} (\T)$. Figure \ref{fig:criticallinenterm} gives a visual interpretation of Proposition \ref{prop:giveskappa} on $(1/p,s)$-diagrams.

\begin{figure}[h!]  
\centering 
\begin{subfigure}[b]{0.40\textwidth}
\begin{tikzpicture}[x=2cm,y=2cm,scale=0.34]

\draw[thick,, ->] (0,-0.5)--(3,-0.5) node[circle,right] {$\frac{1}{p}$} ;
\draw[thick, ->] (0,-3/2)--(0,2) node[circle,above] {$s$} ;
\draw[thick, color=green] (1,-0.5)--(2.2,0.7) node[circle,above] {};
\draw[thick, color=red, ->] (2.2,0.7)--(3,1.5) node[circle,above] {};

\draw[dashed, color=black] (2.2,-0.5)--(2.2,0.7) node[circle,above] {};

\draw[ thick,color=black]
(-0.05,-0.5) -- (0.05,-0.5)  node[black,left] { $s_0$};

\draw[ thick,color=black]
(1,-0.55) -- (1,-0.45)  node[black,below] { $\frac{1}{p_0}$};
\draw[ thick,color=black]
(2.2,-0.55) -- (2.2,-0.45)  node[black,below] { $\frac{1}{p(f)}$};

\end{tikzpicture}
\end{subfigure}
\begin{subfigure}[b]{0.40\textwidth}
\begin{tikzpicture}[x=2cm,y=2cm,scale=0.34]

\draw[thick,, ->] (0,-0.5)--(3,-0.5) node[circle,right] {$\frac{1}{p}$} ;
\draw[thick, ->] (0,-3/2)--(0,2) node[circle,above] {$s$} ;
\draw[thick, color=green,->] (1,-0.5)--(3,1.5) node[circle,above] {};

\draw[ thick,color=black]
(-0.05,-0.5) -- (0.05,-0.5)  node[black,left] { $s_0$};

\draw[ thick,color=black]
(1,-0.55) -- (1,-0.45)  node[black,below] { $\frac{1}{p_0}$};

\end{tikzpicture}
\end{subfigure}

\caption{$(1/p,s)$-diagram with a special emphasis on the line $\frac{1}{p} \mapsto \frac{1}{p} - \frac{1}{p_0} + s_0$. Green color means that we are in the corresponding Besov spaces, while the red color means that we are not.
On the left, we have that $\kappa^{p_0,s_0}(f) = \frac{1}{p(f)} - \frac{1}{p_0}$. On the right, we have that $\kappa^{p_0,s_0}(f) = \infty$.}
\label{fig:giveskappa}
\end{figure}
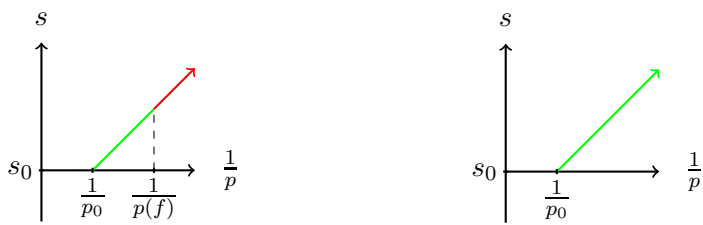

    The result is not new but it has never been stated in that way, that appears to be useful in practice, especially for the connection with the critical smoothness function (see below). 

    \begin{proof}[Proof of Proposition \ref{prop:giveskappa}]
    (i) This is a reformulation of~\cite[Theorem 3]{Fageot2017nterm} (which uses crucially~\cite{Garrigos2004sharp}). Indeed, the relation \eqref{eq:fin} implies, due to \cite[Eq. (47)]{Fageot2017nterm}, 
    that 
    \begin{equation}
    \sigma_n^{p_0,s_0}(f) \leq \frac{C(f)}{n^{\frac{1}{p} - \frac{1}{p_0}} }    
    \end{equation}
    for some constant $C(f)$ independent from $n \geq 1$,
    and therefore  $\kappa^{p_0,s_0}(f) \geq \frac{1}{p} - \frac{1}{p_0}$. This is true for any $p > p(f)$, hence $\kappa^{p_0,s_0}(f) \geq \frac{1}{p(f)} - \frac{1}{p_0}$.
        On the other hand, the relation \eqref{eq:fnotin} implies that~\cite[Eq. (48)]{Fageot2017nterm} does not hold, and therefore $\kappa^{p_0,s_0}(f) \leq \frac{1}{p} - \frac{1}{p_0}$ for every $p < p(f)$. Hence, $\kappa^{p_0,s_0}(f) \leq \frac{1}{p(f)} - \frac{1}{p_0}$. Finally, we have shown \eqref{eq:kappafini}. 

    (ii) As for (i), the relation $f \in B_{p}^{\frac{1}{p}- \frac{1}{p_0} + s_0}(\T)$ implies that $\kappa^{p_0,s_0}(f) \geq \frac{1}{p} -\frac{1}{p_0}$. This being true for any $p > 0$, we deduce that $\kappa^{p_0,s_0}(f) = \infty$

    \end{proof}
    
    \subsection{Compressibility and Critical Smoothness Function}
    
   The function $\frac{1}{p} \mapsto s_f(p)$ is right differentiable, $1$-Lipschitz, and concave. We can therefore define 
   \begin{equation} \label{eq:deftau0}
       \alpha_0(f) = \lim_{p\rightarrow 0} \partial_+ s_f (p),
   \end{equation}
   where we recall that $ \partial_+ s_f (p)$ is the right derivative of $\frac{1}{p} \mapsto s_f(p)$ in $\frac{1}{p}$. The limit \eqref{eq:deftau0} is well-defined because $ \partial_+ s_f (p)$ is bounded below and decreases when $p\rightarrow 0$ and we have $\alpha_0(f) \in [0,1]$.
   
   \begin{theorem} \label{theo:compress}
 Let $0<p_0\leq \infty$, $s_0\in \R$, and  $f \in \SpT$ be such that $s_f(p_0) > s_0$.  
     \begin{itemize}[label=\raisebox{0.30ex}{\tiny$\bullet$}]
       \item Assume that $\alpha_0(f) = 1$, then, for any $p > 0$, $s_f(p) = \frac{1}{p} + s_f(\infty)$ and we have
       \begin{equation}
           \kappa^{p_0,s_0}(f) = \infty. 
       \end{equation}
       
       \item Assume that $\alpha_0(f) < 1$. Then, there exists a unique $p(f)< p_0$ such that $s_f(p(f)) =   \left( \frac{1}{p(f)} - \frac{1}{p_0} \right) + s_0$ and we have 
         \begin{equation} \label{eq:kqppqp0soforpfexisting}
           \kappa^{p_0,s_0}(f) = \frac{1}{p(f)} - \frac{1}{p_0} < \infty.
       \end{equation}
       
        \item If there exists $0 < \tilde{p}(f) < \infty$ such that $\partial_-s_f(\tilde{p} (f))= 1$ and $\partial_+s_f(\tilde{p}(f))= 0$, then 
        \begin{equation}
            \kappa^{p_0,s_0}(f)  = s_f(\tilde{p}(f)) - s_0 = \lim_{p\rightarrow 0} s_f(p) - s_0. 
        \end{equation}
        Moreover, the same conclusion holds if $\partial_+ s_f ( \infty ) = 0$, remarking moreover that $s_f(p) = s_f(\infty)$ for any $0 < p \leq \infty$. 
        \end{itemize} 
   \end{theorem}
   
   \begin{proof}
   Assume first that $\alpha_0(f) = 1$. The concavity of $\frac{1}{p} \mapsto s_f(p)$ implies that, for any $0 < p \leq q \leq \infty$, 
   \begin{equation}
   1 \geq \frac{s_p(f) - s_q(f)}{\frac{1}{p}- \frac{1}{q}} \geq \alpha_0(f) = 1. 
   \end{equation}
   Hence, ${s_p(f) - s_q(f)} = {\frac{1}{p}- \frac{1}{q}}$ for any $p,q$. Taking $q = \infty$ gives that $s_f(p) = \frac{1}{p} + s_f(\infty)$, as expected. 
   
   We have  Moreover that $s_f(p_0) = \frac{1}{p_0} + s_f(\infty) > s_0$. Hence, $s_f(p) > \frac{1}{p} - \frac{1}{p_0} + s_0$, implying that $f \in B_p^{\frac{1}{p}  - \frac{1}{p_0} + s_0}(\T)$ for any $p < p_0$. 
   We therefore have $\kappa^{p_0,s_0}(f) = \infty$ according to the case (ii) in Proposition \ref{prop:giveskappa}. \\

    Assume that $\alpha_0(f) < 1$. Consider the function $h(u) = s_f\left( \frac{1}{u} \right) - \left( u - \frac{1}{p_0} + s_0 \right)$, where the variable $u$ plays the role of $\frac{1}{p}$. Then, $h$ is concave as a sum of a concave and a linear function, and such that $h(p_0^{-1}) > 0$ and $h(u) \rightarrow \infty$ when $u\rightarrow \infty$. Then, $h$ being continuous, there exists $u_0$ such that $h(u_0)=0$. Then, $h(u_0
   ^- ) > 0$ and $h(u_0^+) < 0$ since the left and right derivative of $u$ cannot vanish in $u_0$. This shows that $u_0$ is unique. Setting $p(f) = \frac{1}{u_0}$, we have $s_f(p(f)) = s \left( \frac{1}{p(f)} - \frac{1}{p_0} \right) + s_0$, $p(f)$ being unique because $u_0$ is. Moreover, we have that $\partial_+s_f(p(f)) \leq \partial_- s_f(p(f)) < 1$ (otherwise, we would have that $s_f(p(f)) = s_0 + \frac{1}{p(f)} - \frac{1}{p_0}$, what is excluded). We deduce from $\partial_- s_f(p(f))$ that $s_f(p) > \frac{1}{p}- \frac{1}{p_0} + s_0$ for any $p > p(f)$, hence $f \notin B
  ^{\frac{1}{p} - \frac{1}{p_0} +s_0}_p(\T)$. Similarly, $\partial_+s_f(p(f)) < 1$ implies that $s_f(p) < \frac{1}{p}- \frac{1}{p_0} + s_0$ for any $p < p(f)$, hence $f \in B
  ^{\frac{1}{p} - \frac{1}{p_0} +s_0}_p(\T)$. We are therefore in the situation of Proposition \ref{prop:giveskappa} (i), which implies \eqref{eq:kqppqp0soforpfexisting}.  \\
   
   For the last part, the proposed condition implies that $s_f(p)$ satisfies \eqref{eq:only01slopes} for the critical value $\tilde{p}(f)$. 
   In particular, the unique point for which $s_f(p(f)) = \left( \frac{1}{p(f)} - \frac{1}{p_0} + s_0\right) + s_0$ is such that $s_f(p(f)) =s_f(\tilde{p}(f))$. Moreover, we have that $s_f( p(f) ) = s_f (\tilde{p}(f) ) = \left(\frac{1}{p(f)} - \frac{1}{p_0} \right) + s_0$ and therefore, due to \eqref{eq:kqppqp0soforpfexisting},
   \begin{equation}
       \kappa^{p_0,s_0}(f)  = \frac{1}{p(f)} - \frac{1}{p_0} = s_f(\tilde{p}(f)) - s_0.
   \end{equation}
   Finally, we observe that $s_f(p) = s_f(\tilde{p})$ for any $p < \tilde{p}$, hence $\lim_{p\rightarrow 0} s_f(p) = s_f(\tilde{p})$.  
   \end{proof}
   
   \textit{Remark.} Theorem \ref{theo:compress} provides a simple way to deduce the value of $\kappa^{p_0,s_0}(f)$ from the knowledge of the critical smoothness function: it suffices to find the intersection point of $\frac{1}{p} \mapsto s_f(p)$ and  $\frac{1}{p} \mapsto \frac{1}{p} - \frac{1}{p_0} + s_0$. This is illustrated in the $(1/p,s)$-diagram of Figure~\ref{fig:cavamarcher} with critical smoothness functions of the form $\frac{1}{p} = \frac{\alpha_0}{p} + s_0$.

\begin{figure}[h!]  
\centering 
\begin{subfigure}[b]{0.40\textwidth}
\begin{tikzpicture}[x=2cm,y=2cm,scale=0.34]

\draw[thick,, ->] (0,-0.5)--(3,-0.5) node[circle,right] {$\frac{1}{p}$} ;
\draw[thick, ->] (0,-3/2)--(0,2) node[circle,above] {$s$} ;
\draw[thick, color=blue, ->] (1,-0.5)--(3,1.5) node[circle,above] {};
\draw[thick, color=orange, ->] (0,-0.4)--(3,1.1) node[circle,above] {};

\draw[dashed, color=black] (2.2,-0.5)--(2.2,0.7) node[circle,above] {};

\draw[ thick,color=black]
(-0.05,-0.5) -- (0.05,-0.5)  node[black,left] { $s_0$};

\draw[ thick,color=black]
(1,-0.55) -- (1,-0.45)  node[black,below] { $\frac{1}{p_0}$};
\draw[ thick,color=black]
(2.2,-0.55) -- (2.2,-0.45)  node[black,below] { $\frac{1}{p(f)}$};

\end{tikzpicture}
\end{subfigure}
\begin{subfigure}[b]{0.40\textwidth}
\begin{tikzpicture}[x=2cm,y=2cm,scale=0.34]

\draw[thick,, ->] (0,-0.5)--(3,-0.5) node[circle,right] {$\frac{1}{p}$} ;
\draw[thick, ->] (0,-3/2)--(0,2) node[circle,above] {$s$} ;
\draw[thick, color=blue, ->] (1,-0.5)--(3,1.5) node[circle,above] {};
\draw[thick, color=orange, ->] (0,-0.8)--(3,2.2) node[circle,above] {};

\draw[ thick,color=black]
(-0.05,-0.5) -- (0.05,-0.5)  node[black,left] { $s_0$};

\draw[ thick,color=black]
(1,-0.55) -- (1,-0.45)  node[black,below] { $\frac{1}{p_0}$};

\end{tikzpicture}
\end{subfigure}

\caption{$(1/p,s)$-diagram to determine the wavelet compressibility of $f$. The blue curves have equation $\frac{1}{p}\mapsto \frac{1}{p}-\frac{1}{p_0} + s_0$. The orange curves are the critical smoothness function $\frac{1}{p}\mapsto s_f(p)$ of some functions $f$. On the left, we have that $\kappa^{p_0,s_0}(f) = \frac{1}{p(f)} - \frac{1}{p_0}$. On the right, we have that $\kappa^{p_0,s_0}(f) = \infty$.}
\label{fig:cavamarcher}
\end{figure}
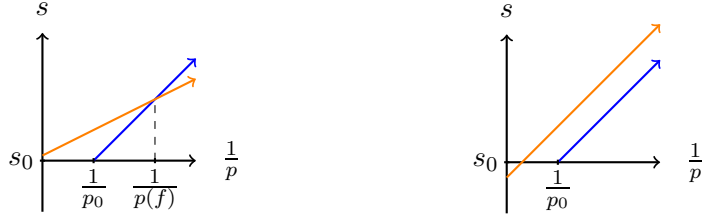

    \section{Discussions and Conclusion}
    \label{sec:conclusion}

    \subsection{Critical Smoothness over Other Function Spaces}\label{sec:beyondbesov}
    
    We defined the critical smoothness of a function with respect to Besov spaces $B_{p,q}^s(\T)$, and we have seen that the critical smoothness is independent from the $q$ parameter (see Proposition~\ref{prop:removeq}). 
    One could consider other function spaces that also depends on a integrability parameter $p > 0$ and a smoothness parameter $s \in \R$, such as Sobolev spaces $W_p^s(\T)$, Bessel-potential spaces $H_p^s(\T)$, or more generally Triebel--Lizorkin spaces $F_{p,q}^s(\T)$~\cite{schmeisser87}.

    Those different spaces are, from the point of view of the critical smoothness, equivalent, in the sense that they specify the same critical smoothness function for any generalized function. In order to make it more precise, let use define the Triebel--Lizorkin critical smoothness associated to $0< p, q \leq \infty$ by
    \begin{equation} \label{eq:cs}
        \tilde{s}_f(p,q) = \sup \{ s \in \R, \ f \in F_{p,q}^s(\T) \} \in (-\infty, \infty].
    \end{equation}
    The quantity $ \tilde{s}_f(p,q) $ is well-defined for the exact same reason than the Besov critical smoothness $ s_f(p,q) $ was (see Section \ref{sec:defspf}). Then, we have the following fact.
    
    \begin{proposition}
    Let $f \in \SpT$, then for any $0< p, q\leq \infty$, we have that
    \begin{equation} \label{eq:TLvsB}
        \tilde{s}_f(p,q) = s_f(p,q) =s_f(p). 
    \end{equation}
    \end{proposition}
    
    \begin{proof}
    Assume first that $p < \infty$. 
    According to~\cite[Proposition 2, p. 47]{schmeisser87}, we have the topological embeddings 
    \begin{equation}
        B_{p,\min(p,q)}^s (\T) \subseteq  F_{p,q}^s (\T) \subseteq  B_{p,\max(p,q)}^s (\T) .
    \end{equation}
    This implies that $s_{p, \max(p,q)}(f) \leq \tilde{s}_f(p,q) \leq s_{p, \min(p,q)}(f)$, Moreover, we know from Proposition \ref{prop:removeq} that $s_{p, \max(p,q)}(f) = s_{p, \min(p,q)}(f) = s_p(f)$, which gives \eqref{eq:TLvsB}. 
    \end{proof}
    
    In other terms, the critical smoothness of generalized functions depends on the integrability parameter $p>0$, and not on the type of function spaces (Besov, Sobolev, Bessel-potential, Triebel--Lizorkin, etc.) we consider. This means in particular that one can choose the best function spaces to characterize the critical smoothness of a generalized function. 

    \subsection{Critical Smoothness in the Multivariate Case}

    We introduced the critical smoothness function of $1$-dimensional periodic functions defined over $\T$. Again, this choice was made for the sake of simplicity (especially regarding the wavelet formalism) but there is no conceptual difficulty for higher dimension generalizations. One can consider multivariate functions $f \in \mathcal{S}'(\T^d)$ over the $d$-dimensional torus with $d\geq 1$. Theorem \ref{theo:main} is then generalized as follows (the extension of the definitions to the multivariate case are left to the reader). 
    
    \begin{proposition}
          Let $f \in \mathcal{S}'(\T^d)  \backslash \mathcal{S}(\T^d) $. Then,
    the function $\frac{1}{p} \mapsto s_f(p)$ is an increasing, concave, and $d$-Lipschitz function from $[0,\infty)$ to $\R$.
    
    Moreover, for any function $s : (0,\infty] \rightarrow \R$ such that $\frac{1}{p}\mapsto s(p)$ is increasing, concave, and $d$-Lipschitz, there exists a periodic generalized function $f \in \mathcal{S}'(\T^d)  \backslash \mathcal{S}(\T^d)$ such that $s_f(p) = s(p)$.
    \end{proposition}
    
    The key difference is the $d$-Lipschitzness. This is a classic fact that can be seen in embedding relations between multivariate Besov spaces~\cite[Section 2.3.3]{Edmunds2008function}. For instance, the critical smoothness function of a $d$-dimensional Dirac comb $\Sha = \sum_{\bm{k}\in \Z^d} \delta(\cdot - \bm{k} )$ is $s_{\Sha} (p) = d - \frac{d}{p}$ for every $0<p \leq \infty$~\cite[Proposition 5]{aziznejad2018wavelet}. 
    
    \subsection{Conclusion}
    
    In this paper, we characterized the possible evolution of the critical smoothness $s_f(p)$ of generalized functions $f \in \SpT$ for $0<p \leq \infty$. We have shown that the class of critical smoothness functions $\frac{1}{p} \mapsto s_f(p)$ coincides with the class of functions $\frac{1}{p}\mapsto s(p)$ that are increasing, concave, and $1$-Lipschitz over $[0,\infty)$ for functions $f \notin\ST$.
    We moreover obtained several interesting properties of the critical smoothness functions, with a special emphasis on the characterization of the compressibility of generalized functions in wavelet domain via their critical smoothness functions.
    We hope that the critical smoothness function will be a useful tool for characterizing the smoothness properties of deterministic and random (generalized) functions.

\section*{Acknowledgments}

The authors are grateful to Felix Hummel for interesting discussions regarding interpolation theory.
Julien Fageot was supported by the Swiss National Science Foundation (SNSF) under Grant P2ELP2\_181759. 
\bibliographystyle{plain}
\bibliography{references}

\begin{thebibliography}{10}

\bibitem{andersson1997characterization}
P.~Andersson.
\newblock Characterization of pointwise {H}{\"o}lder regularity.
\newblock {\em Applied and Computational Harmonic Analysis}, 4(4):429--443,
  1997.

\bibitem{aubry2002random}
J.M. Aubry and S.~Jaffard.
\newblock Random wavelet series.
\newblock {\em Communications in Mathematical Physics}, 227(3):483--514, 2002.

\bibitem{ayache2010holder}
A.~Ayache and S.~Jaffard.
\newblock H{\"o}lder exponents of arbitrary functions.
\newblock {\em Revista Matem{\'a}tica Iberoamericana}, 26(1):77--89, 2010.

\bibitem{ayache1999generalized}
A.~Ayache and J.~Lévy V{\'e}hel.
\newblock Generalized multifractional {B}rownian motion: {D}efinition and
  preliminary results.
\newblock In {\em Fractals}, pages 17--32. Springer, 1999.

\bibitem{aziznejad2018wavelet}
S.~Aziznejad and J.~Fageot.
\newblock Wavelet analysis of the {B}esov regularity of {L}\'evy white noises.
\newblock {\em arXiv preprint arXiv:1801.09245}, 2018.

\bibitem{aziznejad2020wavelet}
S.~Aziznejad and J.~Fageot.
\newblock Wavelet compressibility of compound {P}oisson processes.
\newblock {\em arXiv preprint arXiv:2003.11646}, 2020.

\bibitem{barral2007singularity}
J.~Barral and S.~Seuret.
\newblock The singularity spectrum of {L}{\'e}vy processes in multifractal
  time.
\newblock {\em Advances in Mathematics}, 214(1):437--468, 2007.

\bibitem{Blumenthal1961sample}
R.M. Blumenthal and R.K. Getoor.
\newblock Sample functions of stochastic processes with stationary independent
  increments.
\newblock {\em Journal of Mathematics and Mechanics}, 10:493--516, 1961.

\bibitem{bochkina2013besov}
N.~Bochkina.
\newblock Besov regularity of functions with sparse random wavelet
  coefficients.
\newblock {\em arXiv preprint arXiv:1310.3720}, 2013.

\bibitem{Bottcher2014levy}
B.~B{\"o}ttcher, R.L. Schilling, and J.~Wang.
\newblock {\em L{\'e}vy {M}atters III: L{\'e}vy-{T}ype {P}rocesses:
  {C}onstruction, {A}pproximation and {S}ample {P}ath {P}roperties}, volume
  2099.
\newblock Springer, 2014.

\bibitem{chong2019path}
C.~Chong, R.C. Dalang, and T.~Humeau.
\newblock Path properties of the solution to the stochastic heat equation with
  {L}{\'e}vy noise.
\newblock {\em Stochastics and Partial Differential Equations: Analysis and
  Computations}, 7(1):123--168, 2019.

\bibitem{christensen2016introduction}
O.~Christensen.
\newblock {\em An introduction to frames and {R}iesz bases}.
\newblock Springer, 2016.

\bibitem{ciesielski1993orlicz}
Z.~Ciesielski.
\newblock Orlicz spaces, spline systems, and {B}rownian motion.
\newblock {\em Constructive Approximation}, 9(2-3):191--208, 1993.

\bibitem{Ciesielski1993quelques}
Z.~Ciesielski, G.~Kerkyacharian, and B.~Roynette.
\newblock Quelques espaces fonctionnels associ{\'e}s {\`a} des processus
  gaussiens.
\newblock {\em Studia Mathematica}, 107(2):171--204, 1993.

\bibitem{Cioica2015besov}
P.A. Cioica.
\newblock {\em Besov regularity of stochastic partial differential equations on
  bounded Lipschitz domains}.
\newblock Logos Verlag Berlin GmbH, 2015.

\bibitem{cioica2012spatial}
P.A. Cioica and S.~Dahlke.
\newblock Spatial {B}esov regularity for semilinear stochastic partial
  differential equations on bounded {L}ipschitz domains.
\newblock {\em International Journal of Computer Mathematics},
  89(18):2443--2459, 2012.

\bibitem{Cohen2003numerical}
A.~Cohen.
\newblock {\em Numerical analysis of wavelet methods}, volume~32.
\newblock Elsevier, 2003.

\bibitem{dahlke1997besov}
S.~Dahlke and R.A. DeVore.
\newblock Besov regularity for elliptic boundary value problems.
\newblock {\em Communications in Partial Differential Equations},
  22(1-2):1--16, 1997.

\bibitem{dahlke2020besov}
S.~Dahlke, M.~Hansen, C.~Schneider, and W.~Sickel.
\newblock On besov regularity of solutions to nonlinear elliptic partial
  differential equations.
\newblock {\em Nonlinear Analysis}, 192:111686, 2020.

\bibitem{daoudi1998construction}
K.~Daoudi, J.~{L{\'e}vy V{\'e}hel}, and Y.~Meyer.
\newblock Construction of continuous functions with prescribed local
  regularity.
\newblock {\em Constructive Approximation}, 14(3):349--385, 1998.

\bibitem{Devore1998Nterm}
R.A. Devore.
\newblock Nonlinear approximation.
\newblock {\em Acta Numerica}, 7:51--150, 1998.

\bibitem{Edmunds2008function}
D.E. Edmunds and H.~Triebel.
\newblock {\em Function {S}paces, {E}ntropy {N}umbers, {D}ifferential
  {O}perators}, volume 120 of {\em Cambridge Tracts in Mathematics}.
\newblock Cambridge University Press, Cambridge, 2008.

\bibitem{fabian2011banach}
M.~Fabian, P.~Habala, P.~H{\'a}jek, V.~Montesinos, and V.~Zizler.
\newblock {\em Banach space theory: the basis for linear and nonlinear
  analysis}.
\newblock Springer Science \& Business Media, 2011.

\bibitem{fageot2017gaussian}
J.~Fageot.
\newblock {\em Gaussian versus {S}parse {S}tochastic {P}rocesses:
  {C}onstruction, {R}egularity, {C}ompressibility}.
\newblock PhD thesis, Ecole Polytechnique F{\'e}d{\'e}rale de Lausanne, 2017.

\bibitem{Fageot2017multidimensional}
J.~Fageot, A.~Fallah, and M.~Unser.
\newblock Multidimensional {L}{\'e}vy white noise in weighted {B}esov spaces.
\newblock {\em Stochastic Processes and Their Applications}, 127(5):1599--1621,
  2017.

\bibitem{fageot2020tv}
J.~Fageot and M.~Simeoni.
\newblock {TV}-based reconstruction of periodic functions.
\newblock {\em arXiv preprint arXiv:2006.14097}, 2020.

\bibitem{Fageot2017besov}
J.~Fageot, M.~Unser, and J.P. Ward.
\newblock On the {B}esov regularity of periodic {L}{\'e}vy noises.
\newblock {\em Applied and Computational Harmonic Analysis}, 42(1):21 -- 36,
  2017.

\bibitem{Fageot2017nterm}
J.~Fageot, M.~Unser, and J.P. Ward.
\newblock The $n$-term approximation of periodic generalized {L}{\'e}vy
  processes.
\newblock {\em Journal of Theoretical Probability}, 33:180--200, 2020.

\bibitem{Garrigos2004sharp}
G.~Garrig{\'o}s and E.~Hern{\'a}ndez.
\newblock Sharp {J}ackson and {B}ernstein inequalities for {N}-term
  approximation in sequence spaces with applications.
\newblock {\em Indiana University mathematics journal}, 53(6):1741--1764, 2004.

\bibitem{hansen2014n}
M.~Hansen.
\newblock $n$-term approximation rates and {B}esov regularity for elliptic
  {PDE}s on polyhedral domains.
\newblock {\em Preprint}, 131, 2014.

\bibitem{Herren1997levy}
V.~Herren.
\newblock L{\'e}vy-type processes and {B}esov spaces.
\newblock {\em Potential Analysis}, 7(3):689--704, 1997.

\bibitem{hummel2019stochastic}
F.~Hummel.
\newblock {\em Stochastic Transmission and Boundary Value Problems}.
\newblock PhD thesis, Universit{\"a}t Konstanz, 2019.

\bibitem{hummel2020sample}
F.~Hummel.
\newblock Sample paths of white noise in spaces with dominating mixed
  smoothness.
\newblock {\em arXiv preprint arXiv:2005.10858}, 2020.

\bibitem{hytonen2008}
T.~Hyt\"onen and M.C. Veraar.
\newblock On {B}esov regularity of {B}rownian motions in infinite dimensions.
\newblock {\em Probability and Mathematical Statistics}, 28(1):143--162, 2008.

\bibitem{jaffard1995functions}
S.~Jaffard.
\newblock Functions with prescribed {H}{\"o}lder exponent, 1995.

\bibitem{jaffard2000lacunary}
S.~Jaffard.
\newblock On lacunary wavelet series.
\newblock {\em Annals of Applied Probability}, pages 313--329, 2000.

\bibitem{jaffard1996local}
S.~Jaffard and B.B. Mandelbrot.
\newblock Local regularity of nonsmooth wavelet expansions and application to
  the {P}olya function.
\newblock {\em advances in mathematics}, 120(2):265--282, 1996.

\bibitem{jaffard1996wavelet}
S.~Jaffard and Y.~Meyer.
\newblock {\em Wavelet methods for pointwise regularity and local oscillations
  of functions}, volume 587.
\newblock American Mathematical Soc., 1996.

\bibitem{Kabanava2008tempered}
M.~Kabanava.
\newblock Tempered {R}adon measures.
\newblock {\em Revista Matem{\'a}tica Complutense}, 21(2):553--564, 2008.

\bibitem{kahane1977coefficients}
J.P. Kahane, Y.~Katznelson, and K.~De Leeuw.
\newblock Sur les coefficients de {F}ourier des fonctions continues.
\newblock {\em CR Acad. Sci. Paris S{\'e}r. AB}, 285(16):A1001--A1003, 1977.

\bibitem{levy1965processus}
P.~L{\'e}vy and M.~Loeve.
\newblock {\em Processus stochastiques et mouvement brownien}.
\newblock Gauthier-Villars Paris, 1965.

\bibitem{lindner2011approximation}
F.~Lindner and S.~Dahlke.
\newblock {\em Approximation and Regularity of Stochastic PDEs}.
\newblock Shaker, 2011.

\bibitem{Mandelbrot1982fractal}
B.B. Mandelbrot.
\newblock {\em The {F}ractal {G}eometry of {N}ature}.
\newblock W. H. Freeman and Co., San Francisco, Californie, 1982.

\bibitem{Mandelbrot1968}
B.B. Mandelbrot and J.W.~Van Ness.
\newblock Fractional {B}rownian motions, fractional noises and applications.
\newblock {\em SIAM {R}eview}, 10(4):422--437, 1968.

\bibitem{massopust2005splines}
P.R. Massopust.
\newblock Splines, {F}ractal {F}unctions, and {B}esov and {T}riebel-{L}izorkin
  spaces.
\newblock In {\em Fractals in Engineering}, pages 21--32. Springer, 2005.

\bibitem{massopust2016fractal}
P.R. Massopust.
\newblock {\em Fractal functions, fractal surfaces, and wavelets}.
\newblock Academic Press, 2016.

\bibitem{massopust2016local}
P.R. Massopust.
\newblock On local fractal functions in {B}esov and {T}riebel--{L}izorkin
  spaces.
\newblock {\em Journal of Mathematical Analysis and Applications},
  436(1):393--407, 2016.

\bibitem{MeyerWaO}
Y.~Meyer.
\newblock {\em Wavelets and Operators}, volume~37 of {\em Cambridge Studies in
  Advanced Mathematics}.
\newblock Cambridge University Press, Cambridge, 1992.

\bibitem{narcowich2006decomposition}
F.J. Narcowich, P.~Petrushev, and J.D. Ward.
\newblock Decomposition of {B}esov and {T}riebel--{L}izorkin spaces on the
  sphere.
\newblock {\em Journal of Functional Analysis}, 238(2):530--564, 2006.

\bibitem{narcowich2006localized}
F.J. Narcowich, P.~Petrushev, and J.D. Ward.
\newblock Localized tight frames on spheres.
\newblock {\em SIAM Journal on Mathematical Analysis}, 38(2):574--594, 2006.

\bibitem{offin1993note}
D.~Offin and K.~Oskolkov.
\newblock A note on orthonormal polynomial bases and wavelets.
\newblock {\em Constructive Approximation}, 9(2-3):319--325, 1993.

\bibitem{Roynette1993}
B.~Roynette.
\newblock Mouvement brownien et espaces de {B}esov.
\newblock {\em Stochastics: An International Journal of Probability and
  Stochastic Processes}, 43(3-4):221--260, 1993.

\bibitem{rudin1964principles}
W.~Rudin.
\newblock {\em Principles of mathematical analysis}, volume~3.
\newblock McGraw-hill New York, 1964.

\bibitem{Taqqu1994stable}
G.~Samorodnitsky and M.S. Taqqu.
\newblock {\em Stable {N}on-{Gaussian} {P}rocesses: {S}tochastic {M}odels with
  {I}nfinite {V}ariance}.
\newblock Stochastic Modeling. Chapman \& Hall, New York, 1994.

\bibitem{Sato1994levy}
K.~Sato.
\newblock {\em L\'evy {P}rocesses and {I}nfinitely {D}ivisible
  {D}istributions}, volume~68.
\newblock Cambridge University Press, Cambridge, 2013.

\bibitem{Schilling1997Feller}
R.L. Schilling.
\newblock On {F}eller processes with sample paths in {B}esov spaces.
\newblock {\em Mathematische Annalen}, 309(4):663--675, 1997.

\bibitem{Schilling1998growth}
R.L. Schilling.
\newblock Growth and {H}{\"o}lder conditions for the sample paths of {F}eller
  processes.
\newblock {\em Probability Theory and Related Fields}, 112(4):565--611, 1998.

\bibitem{Schilling2000function}
R.L. Schilling.
\newblock Function spaces as path spaces of {F}eller processes.
\newblock {\em Mathematische Nachrichten}, 217(1):147--174, 2000.

\bibitem{schmeisser87}
H.-J. Schmeisser and H.~Triebel.
\newblock {\em Topics in {F}ourier {A}nalysis and {F}unction {S}paces}.
\newblock Wiley Chichester, 1987.

\bibitem{seuret2002local}
S.~Seuret and J.~Lévy V{\'e}hel.
\newblock The local {H}{\"o}lder function of a continuous function.
\newblock 2002.

\bibitem{Treves1967}
F.~Tr{\`e}ves.
\newblock {\em Topological {V}ector {S}paces, {D}istributions and {K}ernels}.
\newblock Academic Press, New York-London, 1967.

\bibitem{Triebel2008function}
H.~Triebel.
\newblock {\em Function {S}paces and {W}avelets on {D}omains}, volume~7 of {\em
  EMS Tracts in Mathematics}.
\newblock European Mathematical Society (EMS), Z\"urich, 2008.

\bibitem{Triebel2010theory}
H.~Triebel.
\newblock {\em Theory of {F}unction {S}paces}.
\newblock Modern Birkh\"auser Classics. Birkh\"auser/Springer Basel AG, Basel,
  2010.
\newblock Reprint of 1983 edition [MR0730762].

\bibitem{Veraar2010regularity}
M.C. Veraar.
\newblock Regularity of {G}aussian white noise on the $d$-dimensional torus.
\newblock {\em Marcinkiewicz centenary volume}, 95:385--398, 2011.

\end{thebibliography}

\end{document}